\documentclass[11pt]{amsart}
\usepackage[margin=1.25in]{geometry}

\usepackage{amsmath}
\usepackage{amscd}
\usepackage{enumitem}

\usepackage{bigints}
\usepackage{amssymb}
\usepackage{mathdots}
\usepackage{subcaption}
\usepackage{tikz}
\usepackage{extarrows}
\usepackage[all]{xy}
\usepackage[justification=raggedright]{caption}
\usepackage{verbatim}
\usepackage{mathtools}
\usepackage{bbm}
\usepackage{amsthm}
\usepackage{nccmath}
\usepackage[color=yellow]{todonotes}
\usepackage{pgf}
\usepackage[pscoord]{eso-pic}
\usepackage[colorlinks,citecolor=blue,linkcolor=blue]{hyperref}
\usepackage[nameinlink]{cleveref}
\usepackage{comment}
\usepackage{marginnote}
\usepackage{import}
\usepackage{listings}
\usepackage{color}

\theoremstyle{definition}
\newtheorem{thm}{Theorem}[section]
\newtheorem{conj}[thm]{Conjecture}
\newtheorem{lem}[thm]{Lemma}
\newtheorem{defin}[thm]{Definition}
\newtheorem{prop}[thm]{Proposition}
\newtheorem{cor}[thm]{Corollary}

\newtheorem{question}{Question}

\theoremstyle{remark}
\newtheorem*{rem}{Remark}
\newtheorem{clm}{Claim}[thm]

\newcommand{\im}{\,\mathrm{Im}\,}
\newcommand{\re}{\,\mathrm{Re}\,}

\renewcommand{\Re}{\mathrm{Re}}

\renewcommand{\SS}{\mathbb{S}}

\newcommand{\HH}{\mathbb{H}}

\newcommand{\RR}{\mathbb{R}}
\newcommand{\QQ}{\mathbb{Q}}

\newcommand{\CC}{\mathbb{C}}

\newcommand{\ZZ}{\mathbb{Z}}

\newcommand*\rfrac[2]{{}^{#1}\!/_{#2}}     
\newcommand{\del}{\partial}

\renewcommand{\H}{\mathcal{H}}

\newcommand{\K}{\mathcal{K}}

\newcommand{\C}{\mathcal{C}}

\newcommand{\Periph}{\Gamma_\infty}

\usepackage{multirow}

\definecolor{bettergreen}{rgb}{0,0.6,0.4}
\definecolor{purple}{rgb}{0.4,0,0.6}

\newcommand*{\defeq}{\mathrel{\vcenter{\baselineskip0.5ex \lineskiplimit0pt
                     \hbox{\scriptsize.}\hbox{\scriptsize.}}}%
                     =}

\newtheoremstyle{TheoremNum}
        {10pt}{\topsep}              
        {\itshape}                      
        {}                              
        {\bfseries}                     
        {.}                             
        { }                             
        {\thmname{#1}\thmnote{ \bfseries #3}}
    \theoremstyle{TheoremNum}
    \newtheorem{thmn}{Theorem}
    \newtheorem{corn}{Corollary}

\newcommand{\define}[1]{\textbf{#1}}
\newcommand{\orb}{\mathcal{O}}
\newcommand{\orbQ}{\mathcal{Q}}
\newcommand{\thin}{< \epsilon}
\newcommand{\thick}{\geq \epsilon}
\newcommand{\thinD}{< \epsilon/d_{L}}
\newcommand{\thickD}{\geq \epsilon/d_{L}}

\newcommand{\thickDN}{\geq \epsilon/d_{N}}

\newcommand{\horopack}{\mathcal{H}}
\renewcommand{\H}{\horopack}
\newcommand{\nerve}{\mathcal{N}}

\newcommand{\mat}[4]{\bigl(\begin{smallmatrix} #1 & #2 \\ #3 & #4\end{smallmatrix}\bigr)}
\newcommand{\I}{\mathfrak{i}}
\renewcommand{\line}{\mathcal{L}}
\newcommand{\prj}{\mathcal{P}}
\newcommand{\hatC}{\widehat{\CC}}

\renewcommand{\C}{\mathfrak{c}}
\newcommand{\Isom}{\mathrm{Isom}}
\renewcommand{\L}{N}
\renewcommand{\K}{M}

\newcommand{\edHT}{(\epsilon,d_L)\mbox{-twisted}}
\newcommand{\edHTN}{(\epsilon,d_N)\mbox{-twisted}}
\newcommand{\SmL}{\SS^3 \setminus L} 
\newcommand{\SmK}{\SS^3 \setminus K} 

\newcommand{\NL}{\mathcal{L}}

\usepackage{color, colortbl}

\definecolor{Gray}{gray}{0.9}

\title{Symmetries and Hidden Symmetries of $(\epsilon, d_L)$-Twisted Knot Complements}

\author{Neil R Hoffman}
 \address{Department of Mathematics, Oklahoma State University, Stillwater, OK}
\email[]{neil.r.hoffman@okstate.edu}

\author{Christian Millichap}
 \address{Department of Mathematics, Furman University, Greenville, SC}
\email[]{christian.millichap@furman.edu}

\author{William Worden}
 \address{Department of Mathematics, Rice University, Houston, TX}
\email[]{william.worden@rice.edu}

\begin{document}
\maketitle
\date\today

\begin{abstract}
In this paper we analyze symmetries, hidden symmetries, and commensurability classes of $\edHT$ knot complements, which are the complements of knots that have a sufficiently large number of twists in each of their twist regions. These knot complements can be constructed via long Dehn fillings on fully augmented link complements. 

We show that such knot complements have no hidden symmetries, which implies that there are at most two other knot complements in their respective commensurability classes. Under mild additional hypotheses, we show that these knots have at most four (orientation-preserving) symmetries and are the only knot complements in their respective commensurability classes. Finally, we provide an infinite family of explicit examples of  $\edHT$ knot complements that are the unique knot complements in their respective commensurability classes obtained by filling a fully augmented link complement with four crossing circles.
\end{abstract}

\section{Introduction}

Two manifolds are said to be \define{commensurable} if they share a common finite-sheeted cover. In the case of hyperbolic 3-manifolds, this property is an equivalence relation and the equivalence classes are called \define{commensurability classes}. In general, commensurability classes of hyperbolic 3-manifolds are not well-understood, and it is often difficult to decide if two such manifolds are commensurable. To make the commensurability problem more tractable, it is natural to restrict to the case of hyperbolic knot complements. Evidence from the literature suggests that hyperbolic knot complements are rarely commensurable with one another. More precisely, Reid and Walsh put forth the following conjecture:

\begin{conj}[{\cite[Conjecture 5.2]{RW08}}]\label{conj:NR}
There are at most three hyperbolic knot complements in a commensurability class.
\end{conj}

Much work has been done to verify the conjecture for particular classes of knot complements, which will be discussed below.
Perhaps the biggest step forward toward a resolution comes from a result of Boileau, Boyer, Cebanu, and Walsh \cite{BBCW12}, which asserts that the conjecture holds for a hyperbolic knot complement $\SmK$ that does not admit hidden symmetries. A \define{hidden symmetry} of a (finite volume) hyperbolic $3$-manifold $M$ is a symmetry of a finite-sheeted cover of $M$ that does not come from a symmetry of $M$. For hyperbolic knot complements, admitting hidden symmetries is equivalent to non-normally covering an orbifold \cite[Proposition 9.1]{NR92}, and we will frequently make use of this perspective. There are only three hyperbolic knot complements known to admit hidden symmetries: the figure-8 knot complement and the two dodecahedral knot complements of Aitchison and Rubinstein \cite{AitchisonRubinstein92}. One of these dodecahedral knots is shown in \Cref{fig:dodeca}. Neumann and Reid have conjectured that these are the only hyperbolic knot complements admitting hidden symmetries. While confirming this conjecture in full generality has proven difficult, certain special (infinite) classes of knot complements have been shown to not have hidden symmetries: $2$-bridge knot complements \cite{RW08}, \cite{MW16}, certain highly-twisted pretzel knot complements \cite{Mil17}, and $(-2,3,n)$-pretzel knot complements with $n \neq 7$ \cite{MaMa08}, as well as certain classes of knot complements that arise from surgery on a common manifold \cite{Hoffman2010}, \cite{CDM19}. 

\begin{figure}[h]
	\begin{subfigure}{.32\textwidth}
 		\raggedright
   		\includegraphics[scale=2]{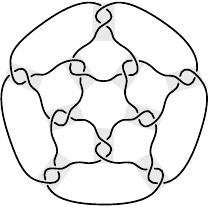}
   		\caption{}
   		\label{fig:dodeca}
	\end{subfigure}
	\begin{subfigure}{.32\textwidth}
 		\centering
   		\includegraphics[scale=1.9]{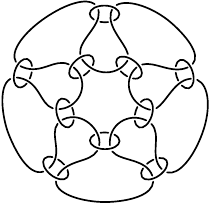}
   		\caption{}
   		\label{fig:dodeca_FAL}
	\end{subfigure}	
	\begin{subfigure}{.32\textwidth}
 		\raggedleft
   		\includegraphics[scale=2]{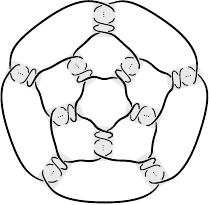}
   		\caption{}
   		\label{fig:dodeca_twisted}
	\end{subfigure}	
	\caption{The dodedecahedral knot on the left is one of only three knots whose complements are known to have hidden symmetries. The twist regions of this knots are highlighted in gray. By augmenting each twist region with a crossing circle and reducing the number of twists modulo 2 we get the FAL in center. If we perform a $\frac{1}{q_i}$-Dehn surgery along each circle $C_i$ of the FAL, we get the knot on the right, with $2q_i$ crossings in each twist region. It is a consequence of \Cref{thm:no_rigid_cusps} that if each $|q_i|$ is sufficiently large, the complement of the knot on the right will have no hidden symmetries.}
	\label{fig:hs_knots}
\end{figure}

With an eye toward expanding this analysis to a far broader class of hyperbolic knot complements, we consider knot complements obtained as Dehn fillings of fully augmented link (FAL) complements. To construct an FAL $L$, start with a (twist-reduced) diagram of a knot $K$, insert a trivial component encircling each twist region (called a crossing circle), and reduce modulo 2 the number of crossings in each twist region. Conversely, $K$ can be recovered from $L$ by Dehn surgery along the crossing circles of $L$, for appropriately chosen surgery slopes (see \Cref{sec:background} for details);  {for example consider the filling used to obtain \Cref{fig:dodeca}} from  \Cref{fig:dodeca_FAL}. In fact, if we choose for each crossing circle component $C_i$ of $L$ \emph{any} slope of the form $\frac{1}{q_i}$, for any $q_i\in\ZZ_{\neq 0}$, the result will be a knot, which by the 6 Theorem will be hyperbolic assuming none of the $q_i$ are too small. Thus, we see that any hyperbolic knot $K$ can be obtained via Dehn surgery on an FAL, and a given FAL will be a Dehn surgery ancestor of infinitely many hyperbolic knots. In particular, the three knot complements known to have hidden symmetries can be obtained by surgery on an FAL; see \Cref{fig:dodeca_FAL} for the FAL ancestor of one of the dodecahedral knots.

In addition to furnishing Dehn surgery ancestors for all hyperbolic knots, the class of FALs also has the benefit of some very nice geometric properties. Most notably, FALs admit an explicit geometric decomposition into a pair of right-angled ideal hyperbolic polyhedra, as was first demonstrated by Agol and D. Thurston in the appendix to \cite{Lac04} (see also \Cref{subsec:FALdecomp}). For sufficiently long Dehn surgeries of the crossing circles, the deformation of the geometric structure is minimal and we are able to study many topological and geometric properties of the resulting knot complements in terms of the geometric structures of the FAL complements {(see for example \cite{Pur07, FP07})}. Going forward this will be our main line of attack, and as a result we will be forced to restrict our analysis to hyperbolic knot complements having sufficiently many twists in each twist region, which we call $\edHT$ knot complements. The exact parameters for the number of twists necessary for a knot to be $\edHT$ depends on the geometry of the FAL being filled; see \Cref{subsub:Twistedknots} for details and \Cref{subsec:examples-twisted-generic} for explicit examples.  Our first result shows that such knot complements do not admit hidden symmetries, providing a major step towards analyzing their commensurability classes.

\begin{thm}\label{thm:no_rigid_cusps}
Let $M = \SS^{3} \setminus K$ be an $(\epsilon,d_L)$-twisted  knot complement. Then $M$ admits no hidden symmetries. 
\end{thm}

Although we prefer to delay the rather technical definition of $(\epsilon,d_L)$-twisted until the next section, we will show in \Cref{prop:HT_non-empty} that all sufficiently long $\frac{1}{q_i}$ fillings of  FAL complements along crossing circles have this property.  Thus we immediately get the following corollary about sequences of knot complements geometrically converging to an FAL complement:

\begin{cor}\label{cor:no_geo_limit}
Let $\{ M_i = \SS^3 \setminus K_i\}$ be a sequence of hyperbolic knot complements resulting from filling the crossing circles of an FAL complement $\SS^3\setminus L$, such that each filling slope has length at least $n_i$, with $\lim_{i\to \infty} n_i=\infty$. Then only finitely many $M_i$ have hidden symmetries.
\end{cor}

In light of this, it is perhaps unsurprising that each of the three examples of knot complements known to admit hidden symmetries arise with just two twists per twist region. Geometrically, this means that the hyperbolic structures of these knot complements are not similar to the hyperbolic structures of their ancestor FAL complements. 

By combining Theorem \ref{thm:no_rigid_cusps} with \cite[Theorem 1.4]{BBCW12} we also get the following corollary, which provides further evidence in favor of the conjecture of Reid and Walsh:

\begin{cor}
There are at most three hyperbolic knot complements in the commensurability class of an $\edHT$ knot complement.	
\end{cor}

Work of Margulis \cite{Mar91} shows that any (non-arithmetic) hyperbolic knot complement covers a unique (orientable) minimal volume orbifold $\orb$ in its commensurability class. Theorem \ref{thm:no_rigid_cusps}  implies that when $M=\SS^3\setminus K$ is an $\edHT$ knot complement such a cover must be regular and in particular, $\orb$ must be the quotient of $M$ by its group of orientation-preserving symmetries. Thus it is natural to study the symmetry group of $M$ in order to more fully understand the commensurability class of $M$. Since our definition of an $\edHT$ knot complement (see \Cref{subsub:Twistedknots}) guarantees that the core geodesics introduced under Dehn filling are the shortest geodesics in $M$, any symmetry of $M$ must map this set of geodesics to itself.  This provides a significant restriction on the symmetry groups of $\edHT$ knot complements, which we highlight in the following theorem. For this result we not only need our knot complements to be $\edHT$, but also generic, which means that no isometry of $\SmK$ permutes the core geodesics of the filling solid tori. For instance, if these core geodesics all have different lengths, then this filling is generic. See \Cref{subsub:TandGknots} for more details on this.

\begin{thm}\label{thm:max_size_sym_group}
Let $M = \SS^{3} \setminus K$ be an $(\epsilon,d_L)$-twisted and generic knot complement. Then $M$ has an orientation-preserving symmetry group of order at most $4$.  
\end{thm}

In addition to establishing that such symmetry groups must be very small, we also determine how any non-trivial symmetries must act on the cusp of $M$; see \Cref{cor:free_cusp_symms}.  Furthermore, we note that there are infinite families of $\edHT$ and generic knot complements with symmetry groups of order exactly 4; see \Cref{rem:tight}. Finally, combining \Cref{thm:max_size_sym_group} and \Cref{thm:no_rigid_cusps} shows that the minimal orbifold cover $M \rightarrow \orb$ described in the previous paragraph is at most degree 4; this is stated in \Cref{cor:HTHDwithRigidCusps}. This result makes it quite feasible to directly analyze $\orb$, and thus assist with distinguishing commensurability classes of  $(\epsilon,d_L)$-twisted and generic knot complements.



Futer and Purcell establish that if a twist-reduced diagram of a knot has enough twist regions and enough twists in each twist region, then the associated complement cannot admit exceptional surgeries (see \cite[Corollary 1.8]{FP07}). We refer the reader to \Cref{fig:dodeca_twisted} for a visual of twist regions and \Cref{sec:background} for the definition. We extend this result by obstructing quotients of  $(\epsilon,d_L)$-twisted and generic knot complements with the same properties  from admitting exceptional surgeries. We summarize both our result and those of Futer and Purcell in the following theorem.

\newcommand{\secondSlopeLongText}{Let $\SS^3 \setminus K$ be an $(\epsilon,d_L)$-twisted and generic knot complement admitting a twist-reduced diagram with at least $9$ twist regions  such that each twist region has at least $6$  crossings. Then $\SS^3 \setminus K$ has no non-trivial exceptional fillings, and the quotient $\orbQ$ of $\SS^3 \setminus K$ by its symmetries that act freely on the cusp has no non-trivial exceptional fillings that are good orbifolds.    
}

\begin{thm}\label{lem:second_slope_long}
\secondSlopeLongText
\end{thm}

See \Cref{sect:6Theorem} for the definition of a good orbifold.The fact that our results extend to quotients allows for the following corollary, which asserts that these knot complements are the unique knots in their respective commensurability classes.

\newcommand{\uniqueKnotComp}{Let $\SS^3 \setminus K$ be an $(\epsilon,d_L)$-twisted and generic knot complement admitting a twist-reduced diagram with at least $9$ twist regions  such that each twist region has at least $6$ crossings. Then $\SS^3 \setminus K$ is the only knot complement in its commensurability class.}

\begin{cor}\label{thm:only_knot_comm_class}
\uniqueKnotComp
\end{cor}

%


This paper focuses on knot complements having a large number of crossings in each twist region. It is natural to ask if it is sufficient to have a single twist region with a large number of crossings. Further, one might ask for a universal lower bound on what is meant by ``large." The below conjecture asserts that such a condition should be sufficient for the conclusion of \Cref{thm:no_rigid_cusps}. Proving such a result would be a major step toward resolving \Cref{conj:NR}.

\begin{conj}
Let $\SmK$ be a knot complement and $D$ be a twist-reduced diagram of $K$. If $\SmK$ admits hidden symmetries, there exists universal upper bound $C$ on the number of crossings in a twist region of $D$.
\end{conj}

\subsection{FAL structure results}
In \Cref{sec:hidden_syms}, a significant amount of work  goes into developing structural results for any  hyperbolic FAL $L$ obtained from fully augmenting a knot $K$. While the main purpose of these results is to provide tools for proving 
{the above}
theorems about $(\epsilon,d_L)$-twisted and generic knot complements, they also {highlight} some interesting properties of horoball packings and orbifold covers of FAL complements. In what follows, let $K_0$ be the component of $L$ corresponding to $K$ and suppose there exists an orbifold cover $p: \SS^{3} \setminus L \rightarrow \orb$.   \Cref{sec:order_3} shows that certain FAL complements cannot admit horoball packings with an order $3$-rotational symmetry, which then restricts the geometry of $\orb$. Specifically, if none of the crossing circle cusps of $\SS^{3} \setminus L$ cover a rigid cusp, then the cusp of $\SS^3 \setminus L$ corresponding to $K_0$ cannot cover an $\SS^{2}(3,3,3)$ or $\SS^{2}(2,3,6)$ rigid cusp; see \Cref{subsub:Twistedknots} for more details on rigid cusps and  \Cref{lem:no_ord3} for a precise statement of this fact. In \Cref{sec:order_4}, FAL complements whose horoball packings admit an order $4$-rotational symmetry are analyzed. Such FAL complements exist, though the geometry of their corresponding horoball packings is quite restrictive, as shown in \Cref{prop:fullH}. As a result, any FAL complement $\SS^3 \setminus L$ where the cusp corresponding to $K_0$ covers a $\SS^2(2,4,4)$ rigid cusp must either have every crossing circle cusp covering a rigid cusp or else $\SS^3 \setminus L$  covers a specific orbifold $\orb$ with two $\SS^{2}(2,4,4)$ cusps; see  \Cref{lem:rigid} for a statement of this result and \Cref{fig:D_orb} and \Cref{fig:DL_orb} for descriptions of $\orb$.

\subsection{Paper Organization}

The paper is arranged as follows. In \Cref{sec:background}, we discuss the necessary background on the geometry of FAL complements and  $(\epsilon,d_L)$-twisted knot complements. In particular, we describe how to construct such knot complements via Dehn filling FAL complements, and discuss some of the important covering space properties of these knot complements. In \Cref{sec:hidden_syms}, we provide a careful analysis of the horoball packings of FAL complements and leverage this analysis to understand orbifold covers of FAL complements that come from restricting orbifold covers of $(\epsilon,d_L)$-twisted knot complements. This all assists in proving \Cref{thm:no_rigid_cusps} at the end of \Cref{sec:hidden_syms}. In \Cref{sec:sym}, we further exploit  our understanding of orbifold covers of FAL complements to prove \Cref{thm:max_size_sym_group} by restricting the degree of certain orbifold covers of FAL complements. In \Cref{sect:unique_in_comm_class}, we prove \ref{lem:second_slope_long} and \Cref{thm:only_knot_comm_class}, which builds off of our previous work along with an orbifold version of the 6-theorem. In \Cref{sec:QR}, we construct explicit examples of $\edHT$ and generic knot complements. For each one of these examples, the knot complement in question is the unique knot complement in its commensurability class.

\subsection{Acknowledgements:} The first author was partially supported by grant from the Simons Foundation (\#524123 to Neil R. Hoffman). Part of this project began when the first and second authors were visiting the  Okinawa Institute of Science and Technology during the Geometry and Topology of 3-manifolds workshop. We thank the Institute and the organizers of the workshop for their hospitality. The second author also wishes to thank Rice University and Oklahoma State University for hosting him during this project. We also thank Dave Futer for providing insightful suggestions on quantifying changes in geometry under Dehn surgery. We wish to thank Jason Deblois for a number of discussions, which led to revisions in \Cref{sec:hidden_syms}. Finally, we wish to thank the referee for a number of thoughtful suggestions about the paper, especially for the prompt to extract the relevant details from the final example as a theorem. 

\section{Background}\label{sec:background}

In this section, we describe the hyperbolic knots and links that we will analyze in this paper. The geometric structures of fully augmented link complements are discussed in  \Cref{subsec:FALdecomp}. In  \Cref{subsec:SuffTwisted}, we describe how to construct $(\epsilon,d_L)$-twisted knot complements by performing Dehn fillings on fully augmented link complements. For the remainder of the paper, all manifolds and orbifolds are assumed to be finite volume, orientable, and hyperbolic unless explicitly stated otherwise.  

\subsection{Fully Augmented Link Complements}
\label{subsec:FALdecomp}

Let $K$ be a hyperbolic link with \define{prime}, \define{twist-reduced} diagram $D(K)$, thought of as a 4-valent planar graph with vertices labelled by over- and under-crossing data (see \cite{FP07}, or \Cref{fig:prime_tr} for definitions). An edge of $D(K)$ is called \define{simple} if it is the unique edge connecting its vertices (i.e., it is not part of a multi-edge). We can partition $D(K)$ into \define{twist regions} by cutting it at the midpoint of every simple edge. More plainly, a twist region is a region of the diagram where two strands twist around each other a maximal number of times---in a twist region with $n$ crossings, the strands twist around each other $\frac{n}{2}$ times. Going forward, we will assume that our link $K$ is embedded in $\SS^3=\RR^3\cup \{\infty\}$ so that it's projection onto the $x,y-$plane is the graph $D(K)$, thus allowing us to refer to twist regions of $K$.

\begin{figure}[h]
	\begin{subfigure}{.45\textwidth}
 		\centering
   		\includegraphics[width=.7\textwidth]{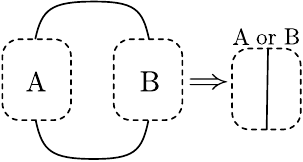}
   		\label{fig:prime}
   		\caption{}
	\end{subfigure}
	\begin{subfigure}{.45\textwidth}
 		\centering
   		\includegraphics[width=.7\textwidth]{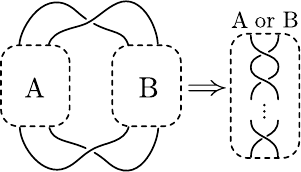}
   		\label{fig:twist_reduced}
   		\caption{}
	\end{subfigure}	
	\caption{Left: A prime link. Right: A twist-reduced link.}
	\label{fig:prime_tr}
\end{figure}

Associated to $K$ is a \define{fully augmented link} $L$, obtained as follows. First, for each twist region $t_i$ of $K$, let $C_i$ be a circle that bounds a twice punctured disk in $\SS^3\setminus K$, punctured by the two strands of $t_i$. Let $L'$ be the disjoint union of $K$ and the circles $C_i$, and let $L$ be the link obtained by reducing modulo 2 the number of crossings in the twist region associated to each circle $C_i$ of $L'$ (see \Cref{fig:to_FAL}). If we consider the complement $\SS^3\setminus L'$, then this reduction modulo 2 is equivalent to removing full twists by cutting along the twice punctured disk bounded by each circle $C_i$, twisting one of the resulting boundary components until at most one crossing remains, then regluing by the identity. Since this operation is a homeomorphism of $\SS^3\setminus L'$, it follows that $\SS^3\setminus L \cong \SS^3 \setminus L'$. Any fully augmented link $L$ constructed in this manner is hyperbolic (see \cite[Theorem 2.2]{FP07} for a short proof of this fact, which follows from \cite[Theorem 4.1]{Ad86}). 

\begin{figure}[h]
	\begin{subfigure}{.32\textwidth}
 		\raggedright
   		\includegraphics[scale=1.3]{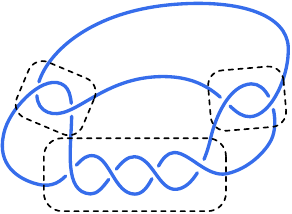}
   		\label{fig:K}
   		\caption{}
	\end{subfigure}
	\begin{subfigure}{.32\textwidth}
 		\centering
   		\includegraphics[scale=1.3]{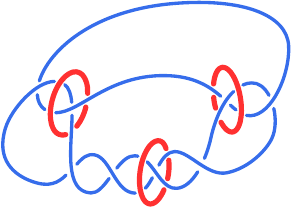}
   		\label{fig:L_prime}
   		\caption{}
	\end{subfigure}	
	\begin{subfigure}{.32\textwidth}
 		\raggedleft
   		\includegraphics[scale=1.3]{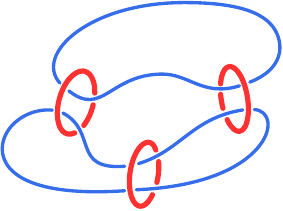}
   		\label{fig:L}
   		\caption{}
	\end{subfigure}	
	\caption{Left: A link $K$ with three twist regions. Center: $L'$ is obtained by augmenting each twist region of $K$ with a circle. Right: the FAL $L$ associated to $K$, obtained by removing full twists from $L'$.}
	\label{fig:to_FAL}
\end{figure}

If $K$ is an $m$-component link with $n$ twist regions, then the fully augmented link (FAL, henceforth) $L$ associated to $K$ will have $m$ \define{planar} components $K_1, \dots, K_m$ (corresponding to the components of $K$, with full twists removed), and $n$ \define{crossing circle} components $C_1,\dots, C_n$. At each crossing circle $C_i$ there is either one crossing, called a \define{half-twist}, or no crossings. A \textbf{planar cusp} of $\SS^3 \setminus L$ is a cusp of this link complement that corresponds with a planar component of $L$. Similary, a \textbf{crossing circle cusp} of $\SS^3 \setminus L$ corresponds with a crossing circle of $L$.

\subsubsection{Polyhedral decomposition} \label{subsec:PD}
In the Appendix of \cite{Lac04}, Agol and D. Thurston give a decomposition of $L$ into two isometric right-angled ideal polyhedra (see also \cite{FP07}, \cite{Pur11}). To simplify the description of this decomposition, we will assume for now that $L$ has no half-twists. Thus each planar component of $L$ can be thought of as a simple closed curve in the projection plane $\prj$, which we identify with $\RR^2\cup \{\infty\} \cong \SS^2\subset \SS^3$. Embed each crossing circle $C_i$ so that the twice punctured \define{crossing disk} it bounds is perpendicular to $\prj$. Then the reflection in $\prj$ preserves $L$, and is therefore a homeomorphism of $\SS^3\setminus L$. By Mostow--Prasad rigidity, this reflection is then an isometry of $\SS^3\setminus L$, and it follows that $\prj\setminus L$ must be totally geodesic in $\SS^3\setminus L$. By a result of Adams \cite{Ad85} the crossing disks also must be totally geodesic surfaces in $\SS^3\setminus L$. By first cutting along $\prj\setminus L$, then along the crossing disks, we obtain two ideal polyhedra $P_1$ and $P_2$. Let $P_1$ be the component that lies above the projection plane. Since we have cut along crossing disks, we can pull apart each half-disk in $P_1$ (see \Cref{fig:to_poly}). By contracting each component of $L$ to a point, we realize $P_1$ as an ideal polyhedron with geodesic faces, 4-valent vertices, and dihedral angles all $\frac{\pi}{2}$. The same is true for $P_2$, as it is isometric to $P_1$. If we shade the faces that come from cutting along disks, and leave projection plane faces unshaded, then the faces of each $P_i$ will be checkerboard colored, and every shaded face will be a triangle.

\begin{figure}[h]
 	\centering
   	\includegraphics[scale=1.1]{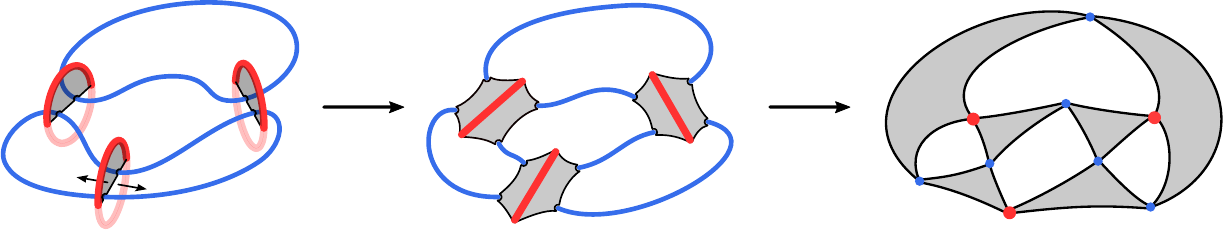}
	\caption{To construct $P_1$, first pull half-disks apart (left), then flatten them onto the projection plane (center), and contract the strands of the link to vertices (right).}
	\label{fig:to_poly}
\end{figure}

If $L$ has a half-twist at one (or more) of its crossing circles, then it still decomposes into polyhedra $P_1$ and $P_2$. The only difference is that when we glue $P_1$ to $P_2$ to recover $\SS^3\setminus L$, we must glue the punctured disks at that crossing circle with a half twist. This means that the corresponding pair of shaded faces of $P_1$ will glue to shaded faces of $P_2$.

The above discussion shows that there are many totally geodesic properly embedded surfaces in an FAL complement, including several thrice-punctured spheres. In the example shown in \Cref{fig:to_poly}, for example, there are 3 obvious thrice-punctured spheres (the crossing disks), and one that is less apparent. In particular, the middle crossing circle bounds a thrice-punctured sphere on both sides (one contains the vertex at $\infty$). We can see this extra thrice-punctured sphere in $P_1\cup P_2$ as follows: a closed path that crosses three unshaded faces, but is not homotopic into the boundary of a shaded face, is the boundary of a triangle embedded in $P_1$, and another in $P_2$. The union of these triangles is a thrice-punctured sphere. Note that if we isotope the middle crossing circle in the left frame of in \Cref{fig:to_poly} so that it wraps around the other planar component, then the ``extra" thrice-punctured sphere is now a crossing disk. In general, if a crossing circle for $L$ bounds multiple twice-punctured disks (whose punctures are meridians of planar components), then such an isotopy exists for each one.  For this reason we will call any such disk a \define{generalized crossing disk}.

\subsubsection{Horoball packing} \label{subsec:HP}

Let $\orb = \HH^3\setminus \Gamma$ be a hyperbolic orbifold with covering map ${\pi:\HH^3\to \orb}$. Given a cusp $\C$ of $\orb$, a \define{cusp neighborhood} is a neighborhood $n(\C)$ of $\C$ such that $\pi^{-1}(n(\C))$ is a union of horoballs. An embedded system of neighborhoods $\{n(\C_i)\}_i$ for the cusps of $\orb$ is called a \define{cusp expansion}. Given a cusp expansion for $\orb$, the union $\bigcup_i \pi^{-1}(n(\C_i))$ of the inverse images of the cusp neighborhoods is called a \define{horoball packing} of $\HH^3$. Such a horoball packing is called \define{maximal} if increasing the size of any cusp neighborhood will result in horoballs in the packing whose interiors are not disjoint. 

We are now ready to describe the (preferred) horoball packing associated to an FAL $L$, which will be our main tool for ruling out hidden symmetries in \Cref{sec:hidden_syms}. We start by describing a certain lift of $P_1\cup P_2$ to $\HH^3$. Going forward we will identify $\HH^3$ with the upper half-space model, with coordinates $(z,t)\in \CC\times \RR_{>0}$. We will identify $\del \HH^3$ with the Riemmann sphere $\hatC$.

Let $\tau$ be a shaded face of $P_1$, and lift $\tau$ so that it has its vertices at $0$, $\I$ and $\infty$. Our choice of lift for $\tau$ determines a holonomy representation for $\pi_1(\SS^3\setminus L)$, and a covering map $\pi:\HH^3\to \SS^3\setminus L$. The horoball packing we are interested in is provided by the following theorem of Futer--Purcell: 

\begin{thm}[\cite{FP07}]\label{thm:hbp_FP07}
Let $L$ be an FAL, and $\pi:\HH^3\to \SS^3\setminus L$ the covering map described above. Then there exists a maximal horoball packing $\H$ of $\HH^3$ such that for any edge $e$ in $\pi^{-1}(P_1\cup P_2)$, the \emph{midpoint} of $e$ is at a tangency of two horoballs.
\end{thm}

Here the \define{midpoint} of an edge $e$ is defined to be the intersection point of $e$ and a geodesic perpendicular to $e$ that emanates from the vertex $v$, where $v$ is the vertex that lies across the shaded face that $e$ bounds (see \cite[Figure 8]{FP07}). For example, the vertical edges of the triangle $\tau$ have midpoints at height 1, as can easily be checked. \Cref{thm:hbp_FP07} implies that the boundary of the horoball $H_\infty$ at infinity must be at height 1. 
 
To describe the other features of this horoball packing that will be important going forward, we will need to understand the neighborhood of the cusp at $\infty \in \del \HH^3$. To this end, let $L_0$ be the component of $L$ whose corresponding  cusp in $\SS^3\setminus L$ lifts to $\infty$ under $\pi$. Then $\pi(\del H_\infty)$ is a torus that bounds a neighborhood of this cusp. Let $T_0$ be a lift of this torus to $H_\infty$ so that $T_0$ lies above the first quadrant of $\CC$ and has a corner at $(0,1)\in \HH^3$. To ease exposition, it will be convenient to again restrict to the case where $L$ has no half-twists. In this case we have the following lemma from \cite{FP07}:

\begin{lem}\label{lem:tiled}
$T_0$ is tiled by a grid of rectangles, each of which has sides of length 1 parallel to the imaginary axis which intersect shaded faces of $P_1$ or $P_2$, and sides parallel to the real axis intersecting unshaded faces. If $L_0$ is a planar component of $L$ that passes through $m$ crossing disks (counted with multiplicity), then $T_0$ consists of $2m$ tiles and has meridian of length 2, which is parallel with the imaginary axis. If $L_0$ is a crossing circle, then $T_0$ consists of $2$ tiles and has longitude of length 2, parallel with the imaginary axis. See \Cref{fig:cusp_tiles}.
\end{lem}

\begin{figure}[h]
	\begin{subfigure}{.33\textwidth}
 		\centering
   		\includegraphics[scale=1.2]{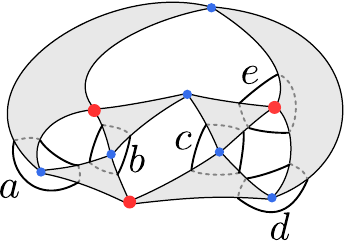}
   		\label{fig:truncation}
   		\caption{}
	\end{subfigure}
	\begin{subfigure}{.40\textwidth}
 		\centering
   		\includegraphics[scale=1.2]{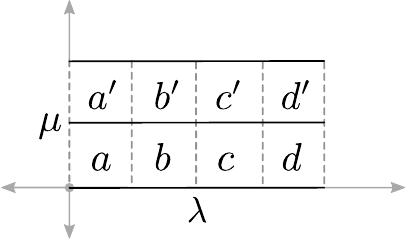}
   		\label{fig:planar}
   		\caption{}
	\end{subfigure}
	\begin{subfigure}{.25\textwidth}
 		\centering
   		\includegraphics[scale=1.2]{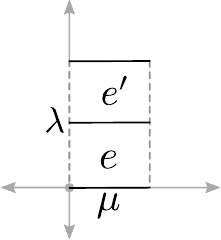}
   		\label{fig:ccircle}
   		\caption{}
	\end{subfigure}	
	\caption{Tilings of a planar cusp (center) and a crossing circle cusp (right). Tiles $a,b,c,d$, and $e$ come from truncation of vertices of $P_1$ as shown (left), while tiles $a',b',c',d'$ and $e'$ come from truncating vertices of $P_2$. Vertices coming from crossing circles are colored red, and larger.}
	\label{fig:cusp_tiles}
\end{figure}

Now let $L_0$ be a planar component, and let the shaded faces that intersect the tiles of $T_0$ lie over the lines $\line_i=\{\Re(z)=l_i\}$ for $0\le i\le m-1$, where $l_0=0$. Then the points $l_i$, $l_i+\I$, and $l_i+2\I$ along the line $\line_i$ are vertices of two shaded faces, and each such vertex is connected by an edge to the vertex at $\infty$. These shaded faces map via $p$ to a crossing disk $D_j$, and $L_0$ must pass through this crossing disk, as it is a vertex of said shaded faces. Since the midpoint of a vertical edge is at height 1, it follows from \Cref{thm:hbp_FP07} that there is a diameter 1 horoball centered at each of these vertices. One of these horoballs (of which there are two up to deck transformation) corresponds to the crossing circle $C_j$ that bounds $D_j$. The other corresponds to the other planar component that passes through $D_j$ (which may be $L_0$ itself). Translating these by the meridian deck transformation, we get a line of \define{full-sized} (i.e., diameter 1) horoballs at the points $l_i+k\I$ for $k\in\ZZ$, which correspond alternately to crossing circle and planar components of $L$ (see \Cref{fig:horo_lines}).

While we restricted to FALs without half-twists to make the statement of \Cref{lem:tiled} more palatable, the consequences of the lemma that we are interested in hold when half-twists are present. In particular, the pattern of lines of alternating color horoballs shown in \Cref{fig:horo_lines} is still present. We leave details to the reader.

\begin{figure}
 	\centering
   	\includegraphics[scale=.95]{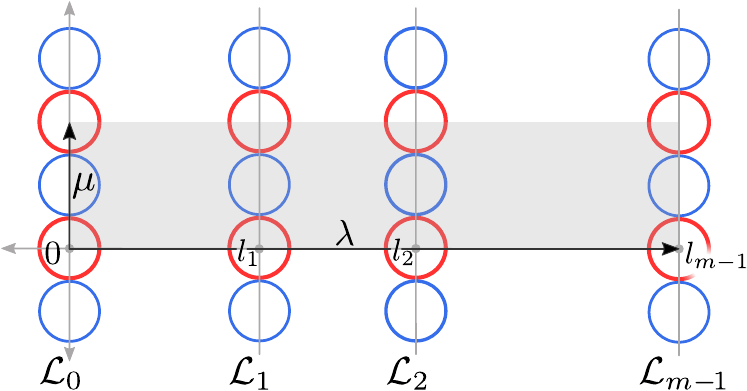}
   	\caption{Looking down from a planar cusp $L_0$ lifted to infinity in $\del \HH^3$, we see lines of full sized horoballs parallel to the imaginary axis of $\CC\subset \del \HH^3$. Each line corresponds to a crossing disk that $L_0$ passes through, and the embedded $\HH^2$ lying above such a line is the developing image in $\HH^3$ of the crossing disk. Horoballs corresponding to crossing circles are colored red, and drawn thicker.}
   	\label{fig:horo_lines}
\end{figure}

Going forward, we will be interested solely in the case where there is a single planar component $K_0$ of $L$. By the horoball packing $\H$, we will always mean the horoball packing obtained from \Cref{thm:hbp_FP07}, where the planar component $K_0$ corresponds to the vertex lifted to $\infty$. In this case, $K_0$ passes through every crossing disk, and every shaded face has two of its vertices corresponding to $K_0$. It follows from the above discussion that for every crossing circle $C_j$ there is a line $\line_{i_j}$ of full-sized horoballs in $\H$ in which the horoballs correspond alternately to $C_j$ and $K_0$. In fact, there are two such lines $\line_{i_j}$. It will be convenient to record the following important consequence:

\begin{lem}\label{lem:tangent}
Let $L$ be an FAL with one planar component $K_0$, and let $\H$ be the horoball packing described above (with $K_0$ at $\infty$). Then for every component of $L$, there is a full-sized horoball in $\H$, necessarily tangent to $H_\infty$, that is a lift of a neighborhood of that component.	
\end{lem}

\subsection{Knot Complements via Dehn Filling FAL Complements} 
\label{subsec:SuffTwisted}

Throughout this section, let $L = K_0 \sqcup C_1 \sqcup \dots \sqcup C_n$ be an FAL with one planar component $K_0$ and $n$ crossing circles $C_i$, $i=1,\dots, n$. In this section and going forward, we will abuse notation and refer to $K_0$ and the $C_i$ both as components of $L$, and as cusps of $\SS^3\setminus L$.

Let $T_i$ be the torus boundary of a neighborhood of a crossing circle $C_i$. An isotopy class of simple closed curve on $T_i$ is naturally identified with a tuple $(p_i,q_i)\in \ZZ \times \ZZ \cong H_1(T_i,\ZZ)$, where we identify the canonical meridian and longitude of $C_i$ with a basis for $H_1(T_i,\ZZ)$. By removing a neighborhood of $C_i$ and gluing in a solid torus along $T_i$ so that its meridian glues to the curve $(p_i,q_i)$, we obtain the $(p_i,q_i)$--Dehn filling of $C_i$. The parameters $(p_i,q_i)$ describes the \define{slope} of the filling. When $\gcd(p_i,q_i)=1$ it is natural to identify the slope of a filling with the rational number $\frac{p_i}{q_i}$.

The main goal of our work is to better understand the geometry and topology of  knot complements obtained from high parameter fillings of FAL complements. In \cite{Pur07}, Purcell studies fillings of FALs that result in at least $c$ crossings per twist region, where $c>0$ is an explicit universal constant. This condition is equivalent to requiring that the Dehn filling slope on each crossing circle $C_i$ is $\frac{1}{q_i}$, for some $q_{i} \geq c/2$, $i =1, \ldots, n$. A knot with this property is often called a highly twisted knot.

Similar to Purcell, we will analyze knots with a large number of twists in each twist region. In some sense, the complement of such a knot is geometrically similar to the corresponding FAL complement, and we can control certain geometric aspects of the knot complement by leveraging the well-understood structure of the FAL complement. Specifically, we want to obtain control over the short geodesics in an FAL filling using the thick/thin decomposition of orbifolds, which we now define. Let $\orb = \HH^3 / \Gamma$ be a finite volume hyperbolic $3$-orbifold. Given $\epsilon > 0$, the \define{$\epsilon$-thick part} of $\orb$ is defined as 

$$
\orb_{\thick} \coloneqq \{ x \in \orb : d(\tilde{x}, \alpha \tilde{x}) \geq \epsilon \hspace{0.05in} \text{for all infinite order} \hspace{0.05in} \alpha \in \Gamma  \hspace{0.05in} \text{and all lifts} \hspace{0.05in} \tilde{x} \in \HH^3 \hspace{0.05in} \text{of} \hspace{0.05in} x \}.
$$
The \define{$\epsilon$-thin part} of $\orb$ is defined to be $\orb_{\thin} \coloneqq \orb\setminus \orb_{\thick}$, the complement in $\orb$ of the $\epsilon$-thick part.

 We  refer the reader to \cite{DunbarMeyerhoff1994} for a more thorough background on the thick/thin decomposition and other geometric properties of hyperbolic $3$-orbifolds.

\subsubsection{$(\epsilon,d_L)$-twisted knot complements} \label{subsub:Twistedknots}

Let $\SS^{3} \setminus K$ be a knot complement obtained by Dehn filling each crossing circle $C_i$ of $\SS^3\setminus L$ along a curve of slope $\frac{1}{q_{i}}$, and let $d_L=4\frac{vol(\SS^3\setminus L)}{v_0}$, where $v_0$ is the volume of the regular ideal tetrahedron. This setup implies that $L$ has one planar component.

\begin{defin}\label{def:edHT}
	We say that $\SS^{3} \setminus K$  is \define{$(\epsilon,d_L)$-twisted}, or alternatively that $\SS^{3} \setminus K$ is an \define{$(\epsilon,d_L)$-twisted filling} of $\SS^3 \setminus L$ , if the following three conditions hold:
\begin{enumerate}[label=(\roman*)]
\item $(\SS^3 \setminus L)_{\thick}$ is homeomorphic to $\SS^3 \setminus n(L)$, 
\item  $(\SS^{3} \setminus K)_{\thick}$ is homeomorphic to $(\SS^3 \setminus L)_{\thick}$,  and
\item  $(\SS^{3} \setminus K)_{\thickD}$ is homeomorphic to $(\SS^3 \setminus L)_{\thick}$. 
\end{enumerate}

\end{defin} 

For a more general version of this definition, see  \Cref{sec:QR}.

In this setting $\SS^{3} \setminus K$ is a knot complement with a suitable number of twists in each twist region, where the number of twists depends on the geometry of $\SS^3 \setminus L$. This distinguishes our $(\epsilon, d_L)$-twisted knots from  Purcell's highly twisted knots, where the number of the twists required is independent of the geometry of $\SS^3 \setminus L$. We note here that with this definition, the figure-8 knot is \emph{not} $\edHT$ for any choice of $\epsilon$. This follows from observing that the systole of the figure-8 is $1.0870..$, while its FAL ancestor the Borromean rings has systole $2.122..$ and volume $7.327..$, so that $\frac{\epsilon}{d_L}\le \frac{(2.122..)}{4 (7.327..)/v_0}<1.0870$, and (iii) cannot hold. Our reason for wanting to rule out the figure-8 knot complement is that it is the only knot complement that is an \emph{arithmetic} manifold, by \cite{Reid91} (see \cite{NR92} for relevant background on arithmetic manifolds). Thus, it follows that any knot complement that is $\edHT$ is necessarily \emph{non-arithmetic}. This fact is important because it allows us to use the following lemma, which along with the proof of \Cref{prop:HT_non-empty} below, explains our choice of the constant $d_L$ in \Cref{def:edHT}.

\begin{lem}\label{lem:min_vol_non_arithmetic}
Let $\orbQ$ be a cusped non-arithmetic orbifold. Then $vol(\orbQ) \geq \frac{v_0}{4}$. 
\end{lem}

\begin{proof}
The statement follows from Adams' census of all cusped hyperbolic orbifolds with volume less than $\frac{v_0}{4}$ \cite[Corollary 6.2.]{Ad92} and Neumann--Reid's observations that each element of this census is arithmetic  \cite{NR92b}.
\end{proof}

In the next proposition, we show that $(\epsilon,d_L)$-twisted fillings exist in abundance, and have a nice orbifold covering property. In what follows, let $\gamma_{i}\subset \SS^3\setminus K$ be the core geodesic of the solid torus introduced by filling the $i^{th}$ crossing circle cusp of $\SS^{3} \setminus L$, for $i=1, \ldots, n$.

\begin{prop}\label{prop:HT_non-empty}
Let $\L=\SS^3 \setminus L$ be an FAL with one planar component and $n$ crossing circle components. Choose some ordering of the $n$ crossing circles. 

\begin{enumerate}
\item There exists $\epsilon>0$, and positive integers $k_1,\dots, k_n$, such that if the $|q_i|\ge k_i$ for all $i$ then filling the $n$ crossing circle cusps along $\alpha=(\frac{1}{q_i},\dots,\frac{1}{q_n})$ results in an $(\epsilon,d_L)$-twisted knot complement.
\item If $\K=\SS^{3} \setminus K$ is a $(\epsilon,d_L)$-twisted filling of $\L$ and   $p:\K \rightarrow \orb$ is an orbifold cover, then there is an orbifold cover $p_N: \L \rightarrow \orbQ$, where $\orbQ \cong \orb \setminus \sqcup_{i=1}^{n} p(\gamma_{i})$.

\end{enumerate}
\end{prop}

The following proof utilizes properties of geometric convergence. For further background on these ideas we refer the reader to \cite[Chapter E]{BP92}, \cite{DunbarMeyerhoff1994}, and Thurston's notes \cite{Thu78}.

\begin{proof}

 Let $\L=\SS^3\setminus L$, and denote by $\L_\alpha$ the Dehn filling of $\L$ along $\alpha=(\frac{1}{q_i},\dots,\frac{1}{q_n})$. 
 
 Part (1): To start, we require our choice of $\epsilon$ to be smaller than the systole length of $\L$. Such a choice of $\epsilon$ immediately satisfies item (i) from the definition of $(\epsilon,d_L)$-twisted.  After making $\epsilon$ smaller if necessary, \cite[Theorem 5.4, $(2) \Rightarrow (3)$]{DunbarMeyerhoff1994} implies that there exists $q_1,\dots,q_n$ such that for $|q_i|\ge k_i$, we have  $(\L_\alpha)_{\thick}$ is homeomorphic to $\L_{\thick}$, showing item (ii) holds. In particular, this implies that  $(\L_\alpha)_{\thin}$ consists of a neighborhood of the planar cusp $K_0$, and a tubular neighborhood of each core geodesic introduced by Dehn filling a crossing circle. Furthermore, \cite[Theorem 5.4, $(2) \Rightarrow (3)$]{DunbarMeyerhoff1994} implies that the lengths of these core geodesics converge to zero, and the lengths of other geodesics in $\L_\alpha$ remain bounded below by $\epsilon$, as the $|q_i|$ go to $\infty$. Thus by making the $k_i$ larger if necessary, we can guarantee that $(\L_\alpha)_{\ge \epsilon/d_L}$ is homeomorphic to $(\L_\alpha)_{\thick}$, while ensuring that (ii) still holds, giving (iii).

Part (2): Let $\K$ be an $(\epsilon,d_L)$-twisted filling of $\L$ as given by Part (1) and let $d_p$ be the degree of the cover $p$. Since $\K$ is non-arithmetic, we have by \Cref{lem:min_vol_non_arithmetic} that $vol(\orb)\ge \frac{v_0}{4}$. Since volume decreases under Dehn filling, it follows that ${d_p =\frac{vol(M)}{vol(\orb)}\le \frac{vol(\SS^3\setminus L)}{\rfrac{v_0}{4}}=d_L}$. Hence $0<\frac{d_p}{d_L} \leq 1$ and we have that $\K_{\thinD}\subset p^{-1}(\orb_{\thinD})\subset \K_{\thin}$. Since $\K$ is $(\epsilon,d_L)$-twisted, its $\epsilon$-thick and $\epsilon/d_{L}$-thick parts are homeomorphic, and so, $p^{-1}(\orb_{\thickD})$ is homeomorphic $\K_{\thickD}$. Let $\orbQ = \mathrm{int}(\orb_{\thickD})$. Since $\K_{\thinD}$ contains only neighborhoods of the core geodesics $\gamma_i$ and of the planar cusp, it follows that $\orbQ\cong \orb\setminus \sqcup_{i=1}^n p(\gamma_i)$. Since $\K$ is $\epsilon/d_L$-twisted, $\K_{\thickD}$ is homeomorphic to $\SS^3 \setminus n(L)$. Thus we get a cover $p_{\L}:\L \to \orbQ$, summarized below:

\begin{figure}[ht]
$$
\xymatrix{
\L \ar@{.>}[dr]_{p_N} \ar@{>}[r]^{\cong\qquad\,\,\,} & \mathrm{int}(\K_{\thickD}) \ar@{>}[r]^\cong & p^{-1}(\mathrm{int}(\orb_{\thickD})) \ar@{->}[d]^p\\
& \orbQ & \mathrm{int}(\orb_{\thickD}) \ar@{>}[l]_\cong}
$$
\caption{\label{fig:edL_coverings} A schematic for covers used in the proof of \Cref{prop:HT_non-empty}. We stress that with respect to $p_N$ no geodesic in the $\epsilon$-thick part of $M$ can cover a geodesic in the $\epsilon/d_L$-thin part of $\orbQ$.}
\end{figure}
\end{proof}

\begin{rem}
In \Cref{sec:QR}, we produce a quantified version of \Cref{prop:HT_non-empty} by implementing recent tools developed by Futer--Purcell--Schleimer \cite{FPS19}. We also discuss how some of the dependencies on the geometry of $N$ can be simplified if we restrict to filling certain subclasses of FAL complements. 
\end{rem}

The relevance of the proposition above to this paper is made clear by the following theorem of Neumann and Reid, which relates hidden symmetries of a knot complement to covering a rigid cusp orbifold. Recall that there are 3 types of Euclidean turn-overs: $\SS^2(2,3,6)$, $\SS^2(2,4,4)$, and $\SS^2(3,3,3)$. An orientable cusped orbifold is said to have a \define{rigid cusp} if it has a cusp with cross section of Euclidean turn-over. Otherwise, an orientable cusped orbifold is said to have non-rigid cusps. There are two types of non-rigid cusps: pillow-cases, which are homeomorphic to $\SS^{2}(2,2,2,2)$, and tori. 

\begin{thm}[{\cite[Proposition 9.1]{NR92}}]\label{thm:NR}
Let $\SmK$ be a hyperbolic knot complement. Then $\SmK$ admits hidden symmetries if and only if $\SmK$ covers a rigid cusped orbifold.
\end{thm}

In the context of \Cref{prop:HT_non-empty}, this theorem tells us that if we can show that $\orb$ does not have a rigid cusp, then it follows that $M$ does not admit hidden symmetries. The proposition allows us to pass to the orbifold $\orbQ$, and show instead that the corresponding cusp of $\orbQ$ is non-rigid. To show this, we will need to use the fact that the other cusps of $\orbQ$ must also be non-rigid, which follows from the fact that rigid cusps cannot be Dehn filled. This fact is discussed below in \Cref{sub:orbDehnFilling}, where we also give a detailed description of fillings of non-rigid cusps by solid torus quotients. Although such a description is really only needed as background for the proof of \Cref{lem:no_knot_covers}, we include it here since it has the added benefit of providing a description of the thin part $\orb_{\thin}$ of $\orb$ (recall from the proof of \Cref{prop:HT_non-empty} that $\orbQ$ is obtained by removing $\orb_{\thin}$ from $\orb$). Components of $\orb_{\thin}$ consist of the quotients of solid tori described below. We note that the core geodesics $p(\gamma_i)$ of $\orb$ correspond in the description below to the curves and arcs of $n$-torsion in the solid torus quotients (where we may have $n=1$).

\subsubsection{Orbifold Dehn filling}\label{sub:orbDehnFilling}

An important distinction between rigid and non-rigid cusps is that non-rigid cusps have well-defined slopes. We include the discussion here to keep the paper self-contained and to set ideas and notation used later (see \cite[Section 4]{DunbarMeyerhoff1994} for further discussion of orbifold Dehn filling).

First, consider a torus cusp $T$ of an orbifold $\orb$. Such a cusp is filled using an \define{orbi-torus} $D^2(n) \times \SS^1$, where for $n=1$ we just have the solid torus $D^2\times \SS^1$. Alternatively, we could consider an orbi-torus as the quotient of a solid torus by an order-$n$ rotation about its core. Associated to a primitive element $(p,q)\in H_1(T,\ZZ)$ is a simple closed curve $\alpha$. If we glue the meridian $\del(D^2(n)\times \{0\})$ of an orbi-torus to $\alpha$, then we say that the fillings slope is $(np,nq)$. Thus the filling slope describes a multi-curve ($n$ copies of $\alpha$) that becomes trivial in the orbifold fundamental group of the filled orbifold.

For an $\SS^2(2,2,2,2)$ cusp $P$ we fill using the order-2 quotient of an orbi-torus, which we call an \define{orbi-tangle}. This is a 2-strand rational tangle in a ball, with each strand labeled by 2-torsion, and an arc labelled by $n$-torsion connecting the strands of the tangle. When $n=1$ this is the 2-fold quotient of a solid torus (with fixed points on the boundary).
A slope on $P$ is defined via a covering map $p: T \rightarrow P$, where $T$ is a torus. Any essential simple closed curve $\alpha$ on $T$ (which we may assume avoids the singularities) will project 
to an essential simple closed curve $\bar{\alpha}$ on $P$. We define a slope for such 
a closed curve in $P$ in terms of the slope it lifts to in $T$. If we are gluing in an orbi-tangle with $n$-torsion on the arc between the tangles, then there is an orbi-disk $D^2(n)$ properly embedded in the orbi-tangle that separates the two tangle strands. The orbi-tangle is glued to $P$ in such a way that the boundary of the orbi-disk glues to $\bar{\alpha}$ on $P$. If the slope of $\bar{\alpha} \subset P$ is $(p,q)$, then the filling slope for such a filling is $(np,nq)$. As with the torus filling, the filling slope describes a multi-curve ($n$ copies of $\alpha$) that becomes trivial after filling. 

On the other hand, it is not possible to (non-trivially) fill a rigid cusp of an orbifold. This is because no solid-torus quotient has a turnover as its boundary. 
 We record this important fact in the following lemma:

\begin{lem} \label{lem:slopes}
Given any cusped (finite volume) hyperbolic $3$-orbifold $\mathcal{O}$, Dehn fillings can be performed only along non-rigid cusps. 
\end{lem}

For a hyperbolic 3-orbifold $\orb$, each slope we Dehn fill along is associated to a parabolic element of the fundamental group. If we inflate a (maximal) cusp neighborhood of a 3-orbifold and use that to determine a horoball packing of $\HH^3$, we define the \define{length} of a slope to be the displacement of the parabolic in the appropriate horosphere. We point out that for any cover of $\orb$ in which a slope lifts, this gives a consistent measure of the length of that slope.

\subsubsection{$(\epsilon,d_L)$-twisted and generic knot complements} \label{subsub:TandGknots}

Generically, one should expect the $\epsilon$-thin part of an orbifold or manifold to be a set of non-isometric tubes and a set of non-isometric cusp neighborhoods. This motivates the following definition:

\begin{defin}
We say that $\SS^3\setminus K$ is an \define{$(\epsilon,d_L)$-twisted and generic} filling of $\SS^3\setminus L$ if it is $\edHT$ and  no isometry of $\SS^3\setminus K$ non-trivially permutes the set of core geodesics of the filling solid tori.
\end{defin}

The above condition is satisfied if $(\SS^3 \setminus K)_{\thin}$ is a set of (disjoint) pairwise non-isometric tubular neighborhoods of geodesics and a single cusp neighborhood. For example, this will be true if the geodesics in the filling solid tori all have different lengths. 
 
 Note that if $\SS^3\setminus L$ has no symmetries that non-trivially permute crossing circle cusps, then \emph{every} $\edHT$ filling is generic in the sense of the above definition. This is because any isometry of $\SS^3\setminus K$ that permutes core geodesics restricts to a symmetry of $(\SS^3\setminus K)_{\thick}$ permuting the corresponding boundary tori, and $(\SS^3\setminus K)_{\thick}$ is homeomorphic to $\SS^3\setminus n(L)$. Since symmetry groups of link complements can be rigorously computed by SnapPy, this allows us to identify certain FALs for which all sufficiently long fillings are $(\epsilon,d_L)$-twisted and generic (see \Cref{subsec:examples-twisted-generic}). Certainly, though, this condition on the symmetry group of $\SS^3\setminus L$ is far from necessary, as the following proposition demonstrates.

\begin{prop}\label{prop:HTHD_exist}
 For any FAL complement $N= \SS^3 \setminus L$ with one planar component, $(\epsilon,d_L)$-twisted and generic fillings exist.
\end{prop}

\begin{proof}
By \Cref{prop:HT_non-empty} there are positive integers $q_1,\dots,q_n$ such that if $|q_i|\ge k_i$, then filling along $\alpha=(\frac{1}{q_1},\dots,\frac{1}{q_n})$ results in an $(\epsilon,d_L)$-twisted knot complement. Choose an ordering of the crossing circle cusps $\{C_1,...,C_n\}$ of $\SS^3 \setminus L$ and let $\alpha = (\frac{1}{q_{1}}, \ldots, \frac{1}{q_{n}})$ be a multi-slope with $|q_i|\ge k_i$ for all $i$. Then any multi-slope $\alpha' = (\frac{1}{q'_{1}}, \ldots, \frac{1}{q'_{n}})$ with each $ |q_i'| \geq |q_i|$ will also produce an $(\epsilon,d_L)$-twisted filling. To guarantee that such a filling is both $(\epsilon,d_L)$-twisted and generic, first fix a filling $\frac{1}{q'_1}$ for any $|q_1'|\ge |q_1|$ on the first crossing circle cusp. Then choose filling slope $\frac{1}{q'_2}$ for $C_2$, with $|q'_2| \geq |q_2|$, so that the core geodesics of the filling tori for $C_1$ and $C_2$ have distinct lengths. This will be true for all but finitely many choices of $q'_2$, since the length of the core geodesic goes to 0 as $q_2'$ goes to infinity. Continuing this procedure, choose the filling slope $\frac{1}{q'_{i}}$ for $C_i$, with $|q'_i| \geq |q_i|$, so that the resulting core geodesic does not have the same length as that associated to any previous filling. Since $|q_i'|\ge k_i$ for all $i$ and the core geodesics of filling solid tori have distinct lengths, the resulting filling is $(\epsilon,d_L)$-twisted and generic.
 
\end{proof}

\begin{rem}
Though \Cref{prop:HTHD_exist} is only an existence statement, the proof shows that given any FAL complement, many of its fillings will be $(\epsilon,d_L)$-twisted and generic. In particular, the above proof did not rely on the choice of ordering of the cusps, and only required omitting finitely many fillings at each step (though this finite number will generally depend on the filling slopes of previously filled cusps). However, we make no formal claim that $(\epsilon,d_L)$-twisted and generic fillings are generic in a precise probabilistic sense, as we prefer to avoid discussions about probability measures that would take us too far from the goals of this paper. 
\end{rem}

An important feature of $(\epsilon,d_L)$-twisted and generic fillings is that we can understand their symmetries in terms of symmetries of $\SS^3 \setminus L$ that map each cusp to itself. Recall that the orientation-preserving symmetry group of a hyperbolic $3$-manifold $M$ consists of the orientation-preserving homeomorphisms of $M$, up to isotopy; we denote this group by $Sym^{+}(M)$.  The following proposition is key for the proof of \Cref{thm:max_size_sym_group}.

\begin{prop} \label{prop:symrelation}
Let $\K$ be an $(\epsilon,d_L)$-twisted filling of $\L=\SS^3 \setminus L$. Then there exists a monomorphism $f: Sym^{+}(\K) \hookrightarrow Sym^{+}(\L)$ such that for all $\beta \in Sym^{+}(\K)$, $f(\beta)$ permutes the crossing circle cusps of $\L$. Furthermore,  $\K$ is $(\epsilon,d_L)$-twisted and generic if and only if, for every $\beta\in Sym^{+}(\L)$, $f(\beta)$ maps every cusp of $\L$ to itself.
\end{prop}

\begin{proof}
Let $\K$ be an $(\epsilon,d_L)$-twisted filling of $\L$ and let $\{ \gamma_{i} \}_{i=1}^{n}$ be the set of core geodesics in $\K$ introduced by Dehn filling the crossing circle cusps of $\L$. Then since $\{ \gamma_{i} \}_{i=1}^{n}$ are the only geodesics  in $\K$ with length less than $\epsilon$, any $\beta \in Sym^{+}(\K)$ must map this set of geodesics to itself. As a result, Kojima \cite[Lemma 5]{K88} gives a monomorphism  ${f: Sym^{+}(\K) \hookrightarrow Sym^{+}(\L)}$. This monomorphism is called \textit{rest} in Kojima's paper, and comes from restricting the domain of any symmetry of $\K$ to $\K_{\thickD}$. Furthermore, Kojima's work implies that any $f(\beta)$ must fix the meridional classes of $C_1 \sqcup \dots \sqcup C_n$, and therefore must permute the cusps $\{C_1,\dots,C_n\}$.

Now, assume $\K$ is $(\epsilon,d_L)$-twisted and generic. Then any $\beta \in Sym^{+}(\K)$ must map each $\gamma_{i}$ to itself, for $i=1, \ldots, n$. Suppose $f(\beta)$ exchanges two cusps of $\L$. By the previous paragraph, we know that $f(\beta)$ must exchange two crossing circle cusps of $\L$. However, since the monomorphism $f$ comes from restricting symmetries of $\K$, the previous sentence implies that there is a symmetry of $\K$ that exchanges some $\gamma_{i}$ and $\gamma_{j}$,  where $i \neq j$, which is a contradiction.  Likewise, if every $f(\beta)$ maps every cusp of $N$ to itself, then we must have that each corresponding $\beta$ maps each core geodesic obtained from Dehn filling to itself, which implies $\K$ is $(\epsilon,d_L)$-twisted and generic.
\end{proof}

\section{Hidden symmetries and FALs}
\label{sec:hidden_syms}

In this section we prove \Cref{thm:no_rigid_cusps}. Recall that by \Cref{thm:NR}, for hyperbolic knot complements having hidden symmetries is equivalent to covering a rigid cusped orbifold. Thus \Cref{thm:no_rigid_cusps} will follow as a direct corollary of the following statement:

\begin{thm}\label{thm:ht_no_rigid_cusps}
Let L be an FAL with a single planar component, and let $\SmK$ be an $(\epsilon,d_L)$-twisted filling of $\SmL$. Then $\SmK$ does not cover a rigid cusped orbifold. 
\end{thm}

We will delay the proof of the above theorem until the end of the section. However, the idea of the proof is to leverage the thick/thin decomposition of $\SmK$ to show that any orbifold covering $\SmK \rightarrow \orb$ restricts to a covering from the thick part of $\SmK$ to the thick part of $\orb$. Under this covering the thick part of $\orb$ has  at least two cusps, but at most one that is a rigid cusp. However, analyzing the embedded thrice punctured spheres in FAL complements shows that such a restricted cover is impossible. 

In this section we continue to restrict to FALs having a single planar component. Let $L=K_0 \sqcup C_1 \sqcup \cdots \sqcup C_n$ be such an FAL, possibly with half-twists at crossing circles. Recall from \Cref{sec:background} that the complement $\SS^3\setminus L$ decomposes as a union $P_1\cup P_2$ of isometric polyhedra. Since $L$ has a single planar component, every shaded face of $P_1\cup P_2$ has a vertex in the cusp $K_0$. We choose a lift of a shaded face $\tau$ of $P_1$ such that the vertices of $\tau$ lift to $0$, $\I$, and $\infty$. We may further require that the vertex of $\tau$ corresponding to $K_0$ is the one that lifts to $\infty$. With this lift fixed, \Cref{thm:hbp_FP07} gives a maximal horoball packing $\H$. For this section, $\H$ will always denote this particular choice of maximal horoball packing, and we will refer to $\H$ as the \define{preferred horoball packing} for the FAL $L$. It follows from the previous section that $\H$ contains as a subset the pattern of full-sized horoballs shown in \Cref{fig:horo_lines}. Recall that a horoball is \define{full-sized} if it has diameter 1, i.e., it is tangent to the horoball at $\infty$.

The proof of the main theorem follows from analyzing the possible rigid cusped orbifold quotients of all FALs with a single planar component. To prove this proposition, we start in \Cref{sec:order_3} by ruling out the possibility that $K_0$ covers an $\SS^2(2,3,6)$ or $\SS^2(3,3,3)$ rigid cusp. Then in \Cref{sec:order_4} we consider the possible orbifold covers where $K_0$ covers an $\SS^2(2,4,4)$ cusp. In both cases our approach is to study possible symmetries of the preferred horoball packing $\H$. The reason for this is made apparent by the below lemma.

\begin{lem}\label{lem:H_orb}
Let $L$ be an FAL with a single planar component $K_0$, and let
 ${p:\SS^3\setminus L \to \orb}$ be an orbifold cover such that no crossing circle covers a rigid cusp. Then $K_0$ covers a rigid cusp in $\orb$ if and only if $\H$ has a rotational symmetry of order 3 or order 4 fixing $\infty$ (depending on the rigid cusp type).	
\end{lem}

\begin{proof}
Suppose that $K_0$ covers a rigid cusp $\C_0$ of $\orb$. Since by assumption no crossing circle covers a rigid cusp, $K_0$ is the only cusp  covering $\C_0$. We will construct a horoball packing $\H_\orb$ covering cusp neighborhoods of $\orb$. First, let $\pi:\HH^3\to \SS^3\setminus L$ be the cover defined in \Cref{subsec:HP}, and let $\pi_\orb:\HH^3\to \orb$ be the composition $p\circ \pi$. Let $n(\C_0)$ be a maximal embedded cusp neighborhood of $\C_0$ (i.e., it contains every other embedded cusp neighborhood of $\C_0$). For the remaining cusps $\C_1,\dots,\C_m$ of $\orb$, let $n(\C_i)$ be the largest embedded cusp neighborhood whose interior is disjoint from $n(\C_0)$, and from $n(\C_j)$ for $j<i$. Now let $\H_\orb=\bigcup_i \pi_\orb^{-1}(n(\C_i))$. Note that $\pi_\orb^{-1}(n(\C_0))\subset (\H\cap \H_\orb)$, and among these lifts of $n(\C_0)$ are the horoball $H_\infty$ at infinity, and a full-sized horoball tangent to $H_\infty$.

Each cusp $\C_i$, $i\ne 0$, is covered by one or more crossing circles. Since $n(\C_i)$ is maximal, it is either tangent to $n(\C_0)$, or to $n(\C_j)$ for some $j\le i$ (or both). First, consider the case where $n(\C_i)$ is tangent to $n(\C_0)$. Then $\pi_\orb^{-1}(n(\C_i))$ contains a horoball tangent to $H_\infty$, necessarily full-sized. Let $C_k$ be any crossing circle covering $\C_i$, and let $n(C_k)\in p^{-1}(n(\C_i))$. The point of tangency between $n(\C_i)$ and $n(\C_0)$ lifts to a point of tangency between $n(C_k)$ and a neighborhood $n(K_0)\in p^{-1}(n(\C_0))$. It follows that there is a full-sized horoball in $\pi^{-1}(n(C_k))$, and this is the largest horoball covering $C_k$. Let $i_0$ be the smallest index so that $n(\C_{i_0})$ is \emph{not} tangent to $n(\C_0)$. Then since $\pi_\orb=p\circ \pi$, it follows that $\bigcup_{i=0}^{i_0-1} \pi_\orb^{-1}(n(\C_i))\subset \H$.

Consider now the horoballs covering $n(\C_{i_0})$. Since they are not tangent to $n(\C_0)$, they cannot be full-sized. If $C_{k_0}$ is a crossing circle covering $\C_{i_0}$, then the horoballs covering $\C_{i_0}$ in $\H_\orb$ are obtained by equivariantly deflating the horoballs in $\H$ covering $C_{k_0}$. But such horoballs in $\H$ are isolated, and hence they must be isolated in $\bigcup_{i=0}^{i_0} \pi_\orb^{-1}(n(\C_i))$, since $\bigcup_{i=0}^{i_0-1} \pi_\orb^{-1}(n(\C_i))\subset \H$. But this is a contradiction. Since $n(\C_{i_0})$ has a point of tangency with some $n(\C_j)$ for $j\le i_0$, the horoball lifts of $\C_{i_0}$ must be tangent to horoball lifts of $n(\C_j)$.

 It follows that for every crossing circle $C_j$, the largest horoball corresponding to $C_j$ is full-sized. By \Cref{lem:tangent}, the same holds for $\H$. Since $\pi_\orb=p\circ \pi$, it follows that we must have $\H=\H_\orb$. Since deck transformations for $\orb$ are necessarily symmetries of $\H_\orb$, $\H_\orb$ must have an order-4 or order-3 rotation fixing $\infty$ (since $\infty$ covers the rigid cusp of $\orb$). Hence $\H$ has an order-3 or order-4 rotation fixing $\infty$.

Conversely, if $\H$ has an order-4 symmetry then $K_0$ covers an orbifold with an order 4 cone point, which must be $\SS^2(2,4,4)$ as this is the only Euclidean 2-orbifold with this property. Similarly, if $\H$ has an order-3 symmetry then $K_0$ must cover an orbifold with an order 3 cone point, i.e., $\SS^2(2,3,6)$ or $\SS^2(3,3,3)$.
\end{proof}

 In the discussions to follow we will denote by $H_z$ the horoball in $\H$ centered at ${z\in \hatC}$. Since we will usually not need to distinguish between horoballs coming from different crossing circles, we will refer to horoballs that lift from crossing circles as red horoballs. Horoballs that are lifts of $K_0$ will be referred to as blue horoballs. We will similarly refer to red and blue vertices of the polyhedra $P_i$ depending on whether the vertex corresponds to a crossing circle or a planar cusp. Finally, we denote by $\nerve$ the nerve associated to $L$, and we again color the edges of the nerve red or blue depending on the color of the vertex of $P_i$ to which they are dual. In the figures we will color these objects accordingly, and all red vertices, edges, and horoballs will also be drawn thicker.

\subsection{Order 3 symmetries}\label{sec:order_3}
In this section we will rule out the possibility that $K_0$ covers a rigid $\SS^2(2,3,6)$ or $\SS^2(3,3,3)$ cusp. By \Cref{lem:H_orb}, this reduces to showing that $\H$ does not admit order-3 symmetries fixing $\infty$. To this end, suppose  that $\H$ does in fact have an order--3 symmetry $\sigma$ fixing $\infty$ and a point $z\in \CC$, which we may assume is a counterclockwise rotation as seen from $\infty$. Recall that $\line_j$ denotes the line of full-sized horoballs along the line $\{\re(z)=l_j\}$, as in \Cref{fig:horo_lines}. The image $\sigma(\line_0)$ is a line of horoballs having slopes $1/\sqrt{3}$. This line of horoballs must intersect each $\line_j$ in a horoball. In particular, $\line_0\cap \sigma(\line_0)$ is some horoball $H_0$. From here forward we will assume that $H_0$ is red, as shown in \Cref{fig:horo_lines}. If $H_0$ is blue then the discussion to follow and the proofs of \Cref{lem:6_vert} and \Cref{lem:no_ord3} go through with the only change being the reversal of colors. Rotating again by $\sigma$ gives another line $\sigma^2(\line_0)$ of horoballs, this time of slope $-1/\sqrt{3}$. It is clear that this line must also intersect $\line_0$ in a red horoball (since $H_0$ is red). After translating the horoballs in \Cref{fig:order_3} by the group generated by $\mu$ and $\sigma^{-1}\mu\sigma$, we get diameter $1$ horoballs tiling $\CC$ in the pattern shown in \Cref{fig:o3_tiling}. It follows that for each $j$ we have $l_j=k_j \sqrt{3}$, for some $k_j\in \ZZ_{>0}$.

\begin{figure}[h]
	\begin{subfigure}{.49\textwidth}
 		\centering
   		\includegraphics[scale=.9]{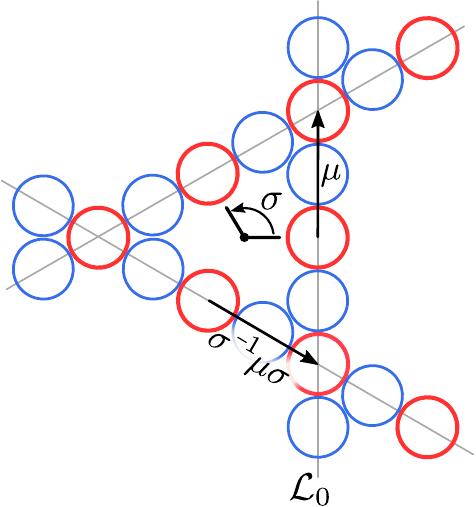}
   		\caption{}
   		\label{fig:order_3}
	\end{subfigure}
	\begin{subfigure}{.49\textwidth}
 		\centering
   		\includegraphics[scale=.9]{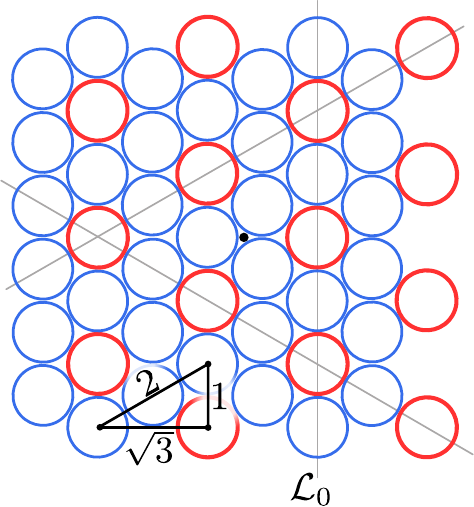}
   		\caption{}
   		\label{fig:o3_tiling}
	\end{subfigure}	
	\caption{The images of a line $\line_0$ of full sized horoballs under an order-3 rotation $\sigma$ by must intersect $\line_0$ in a full sized horoballs, as in (a). By translating the picture in (a) by deck transformations, we get the tiling in (b). As usual, crossing circle horoballs are colored red and drawn thicker.}
	\label{fig:o3}
\end{figure}

\begin{lem}\label{lem:6_vert}
	Suppose $\H$ has an order--3 rotational symmetry fixing infinity. Then every unshaded face of $P_i$, $i=1,2$ has at least $6$ vertices. \end{lem}

\begin{proof}
Let $F$ be an unshaded face of $P_1$. For some $0\le j_0\le m-1$ and $\delta\in\{0,1\}$, there is a lift of $F$ with a vertex at $\infty$, as well as vertices at $a_1=a'_1+\delta \I, \dots, a_n=a'_n+\delta \I$ for some $l_{j_0}=a'_1<a'_2<\dots<a'_n=l_{j_0+1}$. These vertices are centers for a horizontal string of tangent horoballs, since the endpoints of lifted edges of $P_1$ are centers of tangent horoballs.

Recall that $l_j=k_j\sqrt{3}$ for all $j$, so that $|l_{j_0+1}-l_{j_0}|$ can only be a multiple of $\sqrt{3}$. If $|l_{j_0+1}-l_{j_0}|\ne \sqrt{3}$, then it is at least $2\sqrt{3}>3$. It is clear in this case that $n\ge 5$, and hence $F$ has at least $6$ vertices. Hence we may assume that $|l_{j_0+1}-l_{j_0}|=\sqrt{3}$.

We will assume that the horoballs $H_{a_2}$ and $H_{a_{n-1}}$ (centered at $a_2$ and $a_{n-1}$) are as large as possible, without intersecting any horoball in the packing represented in \Cref{fig:o3_tiling}. Then we will show that $H_{a_2}$ and $H_{a_{n-1}}$ are not tangent, so that $3\ne n-1$. Our assumption that $H_{a_2}$ is as large as possible forces it to be tangent to both $H_{a_1}$ and the two blue horoballs $H$ and $H'$ that are tangent to $H_{a_1}$ and $H_{a_n}$ (see \Cref{fig:face_F}) . Hence the center of $H_{a_2}$ is at the barrycenter of the equilateral triangle with vertices the centers of $H_{a_1},H,$ and $H'$. This triangle has height $\frac{\sqrt{3}}{2}$, so the distance $x$ from $a_1$ to the barrycenter satisfies $x^2=(\frac{1}{2})^2+(\frac{\sqrt{3}}{2}-x)^2$, so that $x=\frac{1}{\sqrt{3}}$ (see \Cref{fig:face_F}). Refering to \Cref{fig:pythag_r}, it follows that the radius $r$ of $H_{a_2}$ satisfies $(\frac{1}{2}-r)^2+(\frac{1}{\sqrt{3}})^2=(\frac{1}{2}+r)^2$, so that $r=\frac{1}{6}$. By symmetry the radius of $H_{a_{n-1}}$ is also $\frac{1}{6}$.

Since the center of $a_2$ (resp. $a_{n-1}$) is distance $\frac{1}{\sqrt{3}}$ from $a_1$ (resp. $a_n$) and has radius $\frac{1}{6}$, it follows that if $H_{a_2}$ and $H_{a_{n-1}}$ are tangent then the total distance from $H_{a_1}$ to $H_{a_n}$ must be $2(\frac{1}{6}+\frac{1}{\sqrt{3}})<\sqrt{3}$. But this is a contradiction of our observation that for any $j$, $l_j=k_j \sqrt{3}$, for some $k_j\in \ZZ_{>0}$.
\end{proof}

\begin{figure}[h]
	\begin{subfigure}{.55\textwidth}
		\centering
   		\includegraphics[scale=1.8]{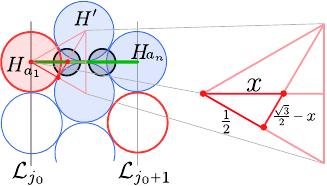}
   		\caption{}
   		\label{fig:face_F}
	\end{subfigure}
	\begin{subfigure}{.44\textwidth}
		\centering
   		\includegraphics[scale=1.8]{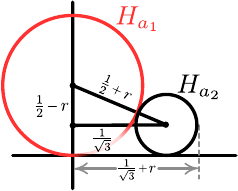}
   		\caption{}
   		\label{fig:pythag_r}
	\end{subfigure}
	\caption{Computing the largest possible radius for $H_{a_2}$ and $H_{a_{n-1}}$.}
	\label{fig:radius}
\end{figure}

\begin{lem}\label{lem:no_ord3}
Let $L$ be an FAL with a single planar component $K_0$, and let ${\pi:\SS^3\setminus L \to \orb}$ be an orbifold cover such that no crossing circle covers a rigid cusp. Then the cusp corresponding to the main component $K_0$ cannot cover an $\SS^2(3,3,3)$ or $\SS^2(2,3,6)$ rigid cusp in $\orb$.	
\end{lem}
\begin{proof}
By \Cref{lem:H_orb} the existence of such a cusp implies that there is an order--3 rotation fixing $\infty$. It follows from the \Cref{lem:6_vert} that if $\H$ has such an order--3 symmetry, then every unshaded face of $P_1$ has at least $6$ vertices. Following \cite[Lemma 2.3]{Pur11}, the nerve $\nerve$ of $P_1$ is a triangulation of $\SS^2$ (with no loops or multi-edges). However, every vertex in this nerve triangulation of $P_1$ is at least $6$-valent, which is impossible as  $\chi(\SS^2)=2$. 
\end{proof}

\subsection{Order 4 symmetries}
\label{sec:order_4}
In this section we investigate order-$4$ symmetries of the preferred horoball packing $\H$, with the goal of showing that $(\epsilon,d_L)$-twisted fillings of FALs cannot cover $\SS^2(2,4,4)$ rigid cusps.  First, though, we remark that a recent pre-print of the first author \cite[Theorem 1.1]{hoffman_cusp_types} obstructs hyperbolic knot complements from covering orbifolds with $(2,4,4)$ cusps, and thus provides an alternate proof to the main result of this section. Nevertheless, we retain this section in order to keep our proof self-contained, and because we prove along the way some structural results of independent interest.

We cannot rule out order--4 rotational symmetries for horoball packings of all FAL complements. In fact, octahedral FAL complements have horoball packings with order--4 symmetries, coming from the symmetries of the tessellation $\HH^3$ by regular ideal octahedra. 

In this section, as in the last, $L=K_0\sqcup C_1 \sqcup \dots \sqcup C_n$ will be an FAL with a single planar component. Suppose that $\H$ has an order--4 symmetry $\sigma$ fixing $\infty$, and that $\sigma$ does not identify any red horoball with any blue horoball. For each fixed $j_0\in \{0,\dots, m-1\}$, $\sigma(\line_{j_0})$ is a horizontal line of horoballs alternating in color, which must intersect each vertical line $\line_j$ in a horoball. In particular, we have a horizontal line of horoballs $\sigma(\line_0)$, which after translating in the meridian direction gives horizontal lines of horoballs along the lines $\im(z)=2k$ or $\im(z)=(2k+1)$ for $k\in \ZZ$. Rotating again by $\sigma$ now gives one of the two horoball patterns shown in \Cref{fig:packings}. (recall that our lift was chosen so that the horoball at $0$ is red).

\begin{figure}
	\begin{subfigure}{.49\textwidth}
 		\centering
   		\includegraphics[scale=1.1]{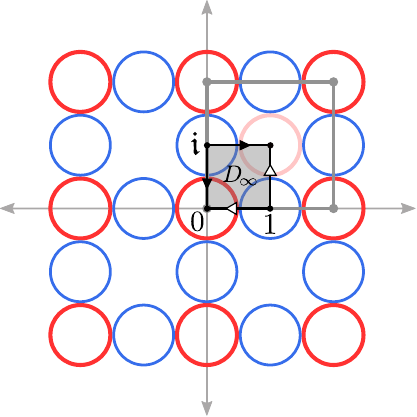}
   		\caption{}
   		\label{fig:packing_even}
	\end{subfigure}
	\begin{subfigure}{.49\textwidth}
 		\centering
   		\includegraphics[scale=1.1]{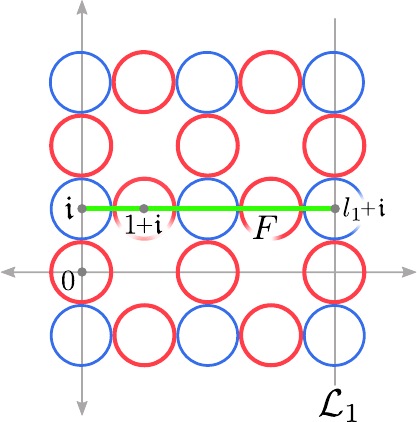}
   		\caption{}   		
   		\label{fig:packing_odd}
	\end{subfigure}	
	\caption{After rotating a line of full-sized horoballs by an order-4 rotation then translating by deck transformations, we get one of the two packings shown above. As usual, crossing circle horoballs are colored red and drawn thicker.}
	\label{fig:packings}
\end{figure}

\begin{figure}
	\begin{subfigure}{.49\textwidth}
 		\centering
   		\includegraphics[scale=1.5]{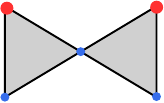}
   		\caption{}
   		\label{fig:local_P_i}
	\end{subfigure}
	\begin{subfigure}{.49\textwidth}
 		\centering
   		\includegraphics[scale=1.5]{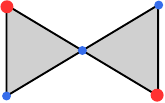}
   		\caption{}
   		\label{fig:bad_tris}
	\end{subfigure}	
	\caption{If all $l_j$ are even, then pairs of shaded faces meeting at a blue vertex must be as in (a). If an unshaded face does not satisfy either condition (1) or condition (2) of \Cref{lem:order4}, then (b) will appear. Red vertices are drawn larger in these figures.}
	\label{fig:tris}
\end{figure}

\begin{lem}\label{lem:order4}
Suppose $\H$ has an order--4 symmetry such that no red horoball is in the orbit of a blue horoball, and all $l_j$ are even. Then every unshaded face of $P_i$ has either (1) all blue vertices, or (2) vertices that alternate in color on a walk around the boundary of the face.	
\end{lem}

\begin{proof}
In either case shown in \Cref{fig:packings}, it follows that pairs of triangular faces of $P_i$ meeting at a blue vertex are as shown in \Cref{fig:local_P_i}. To see this, consider the blue vertex connecting the two triangles in \Cref{fig:local_P_i} as lying at $\infty$, so that the four other vertices are on lines $\line_j$ and $\line_{j+1}$ for some $j$. Since all $l_j$ are even, and since the horoball packing contains one of the patterns shown in \Cref{fig:packings}, all horoballs at points $l_j+b\I$ for fixed $b\in \ZZ$ are full-sized horoballs all of the same color. It follows that both vertices of a unshaded face adjacent to $\infty$ are the same color.


Therefore $P_i$ must be built by joining together pieces like the one shown in \Cref{fig:local_P_i}. Suppose that some unshaded face does not satisfy either (1) or (2). Then along this face there are $3$ consecutive vertices that are blue, blue, red (note that red vertices can never be adjacent to each other). Consider the two triangles opposite the two edges joined by these vertices. Since every triangle contains exactly one red vertex, the only possible configuration of vertices for these two triangles is as in \Cref{fig:bad_tris}, contradicting our assertion that every pair of triangles must have the vertex configuration shown in \Cref{fig:local_P_i}. 
\end{proof}

\begin{prop}\label{prop:fullH}
Let $L = K_0 \sqcup C_1 \sqcup \dots \sqcup C_n$ be an FAL, and let $p: \SS^3\setminus L \to \orb$ be an orbifold cover. Suppose that $K_0$ covers an $\SS^2(2,4,4)$ rigid cusp in $\orb$, and that no crossing circle covers an $\SS^2(2,4,4)$ rigid cusp. Then $H_z$ is a full-sized horoball of $\H$ if and only if $z =a+b\I$ with $a,b \in \ZZ$. Furthermore, \begin{enumerate}
\item all of the cusps corresponding to crossing circles cover a single cusp of $\orb$.
\item all blue full-sized horoballs are centered at $z\in \ZZ[\I]$ with $a+b \equiv 1 \mod 2$.
\item all red full-sized horoballs are centered at $z\in \ZZ[\I]$ with $a+b \equiv 0 \mod 2$.
\item every blue full-sized horoball is centered at a fixed point of an order-4 rotation fixing $\infty$.
\end{enumerate}
\end{prop}

\begin{proof}
Suppose that $K_0$ covers an $\SS^2(2,4,4)$ rigid cusp. Since $K_0$ covers an $\SS^2(2,4,4)$ rigid cusp, there is an order--4 rotational symmetry $\sigma$ fixing infinity, which we may assume is a counter-clockwise rotation as seen from $\infty$. In this case the (maximal) horoball packing $\H$ contains one of the two patterns of horoballs shown in \Cref{fig:packings}, as discussed above. Note that since we are assuming that crossing circles do not cover an $\SS^2(2,4,4)$, the $\sigma$-orbits of blue horoballs and red horoballs are disjoint.

First assume that all $l_j$ are even, and suppose for a contradiction that $\H$ contains the pattern in \Cref{fig:packing_odd}. Referring to the face $F$ highlighted in \Cref{fig:packing_odd} (in which we are looking down on the face), we see that the unshaded face having edges connecting the red vertex at $\infty$ to the blue vertices at $\I$ and $l_1+\I$ also has a red vertex. In particular, the red diameter 1 horoball $H$ centered at $1+\I$ must be at a vertex of $F$. Since $F$ has three consecutive blue vertices, it cannot satisfy either condition of \Cref{lem:order4}. Thus if all $l_j$ are even then $\H$ cannot contain the pattern in \Cref{fig:packing_odd}, so it must contain that of \Cref{fig:packing_even}.

Let $\mu\in\pi_1(\SS^3\setminus L)$ be the lift of the meridian of $K_0$ to the peripheral subgroup at $\infty$, and let $\Periph=\langle \mu,\sigma\rangle$. If some $l_j$ is odd, then the $\Periph$-orbits of $\line_j$ and $\line_0$ are different: one of them looks like \Cref{fig:packing_even} and the other like \Cref{fig:packing_odd}. The union of these orbits looks like \Cref{fig:packing_even} with additional red full-sized horoballs in the empty spaces. Thus $\H$ contains as a subset the pattern in \Cref{fig:packing_even}.

In either case, we find that the horoball packing $\H$ must contain as a subset the horoball packing shown in \Cref{fig:packing_even}. Since $\tau=\sigma^{-1}\mu\sigma$ is a translation by 2 in the longitude direction (or $z \mapsto z+2$), a fundamental domain for $\langle \sigma,\tau\rangle$ is contained in
$$
\{(z,t)\in\HH^3 \mid 0\le Re(z) \le 2, 0\le Im(z) \le 2\}.
$$
The intersection of this domain with $H_\infty$ is shown as a (large) grey square in \Cref{fig:packing_even}. Thus we may assume that the other fixed point $z_\sigma$ of $\sigma$ satisfies $0\le \re(z_\sigma) < 2$ and $0\le \im(z_\sigma)<2$. Since we are assuming that no crossing circle covers an $\SS^2(2,4,4)$ rigid cusp, it follows that either $z_\sigma=1+\I$, or up to equivalence in $\Periph$, it is at one of the two blue horoballs centered at $1$ and $\I$ (since it cannot be centered at a red horoball). If $z_\sigma=1+\I$, then $\mu\sigma$ is a rotation fixing the center of the red horoball at $2\I$, which is not allowed. It follows that $z_\sigma$ is centered on a blue horoball, and in fact by composing with $\mu$ and $\tau$ it is easy to see that every blue horoball shown in \Cref{fig:packing_even} is centered at the fixed point of an order--4 symmetry. This establishes (4). It also follows that there must be a full-sized red horoball centered at $1+\I$, and the centers of full sized horoballs are exactly the lattice $\ZZ[\I]$, with checkerboard coloring as described by (2) and (3). 

To see that all of the crossing circle cusps cover the same cusp of $\orb$, we observe that each crossing circle cusp cobounds a thrice punctured sphere with the cusp $K_0$. Hence, there is a cusp neighborhood of $C_i$ represented by some red full-sized horoball in $\H$. We observe that all red full-sized horoballs are identified by $\Periph$, so that (1) follows.  
\end{proof}

\begin{lem}\label{lem:rigid}
Let $L = K_0 \sqcup C_1 \sqcup \dots \sqcup C_n$ be an FAL, and let $p:\SS^3\setminus L \to \orb$ be an orbifold cover. If $K_0$ covers an $\SS^2(2,4,4)$ rigid cusp in $\orb$, then either each crossing circle $C_i$ also covers an $\SS^2(2,4,4)$ rigid cusp in $\orb$, or $\orb$ is the orbifold in \Cref{fig:DL_orb}.
\end{lem}

\begin{proof}
We will assume 
that no crossing circle covers an $\SS^2(2,4,4)$ rigid cusp in $\orb$, and show that $\orb$ must be the orbifold shown in \Cref{fig:DL_orb}. From \Cref{prop:fullH}, we have that the full sized horoballs of $\H$ are centered above points in the lattice $\ZZ[\I]$ in a checkerboard pattern, and every blue horoball is the center of an order-4 rotation fixing $\infty$.

We will assume from here forward that $\sigma$ is the rotation fixing $H_1$. With this choice of $\sigma$, the action of $\Periph$ on $\HH^3$ has a fundamental domain contained in
$$
D_\infty=\{(z,t)\mid 0 \le Re(z)\le 1\}\cap \{(z,t)\mid 0 \le Im(z)\le 1\}.
$$

We will construct a convex polyhedral fundamental domain $D_\orb$ for $\orb$, which we may assume has an ideal vertex at $\infty$, by further cutting down $D_\infty$. For $z\in \{1,\I\}$, let $S_z$ be the unique half-space of $\HH^3$ containing $H_\infty$, and containing no point of $H_z$. Since $D_\orb$ is an intersection of half-spaces and contains an ideal vertex at $\infty$, it is contained in both $S_1$ and $S_\I$. It follows that $D_\orb\subset D\defeq D_\infty \cap S_1\cap S_\I$, which is a finite volume convex domain having three ideal vertices, one blue (at $\infty$), and two red (at $0$ and $1+\I$). The red vertices are identified under $\Periph$. 

\begin{figure}
 	\centering
   	\includegraphics[scale=1.5]{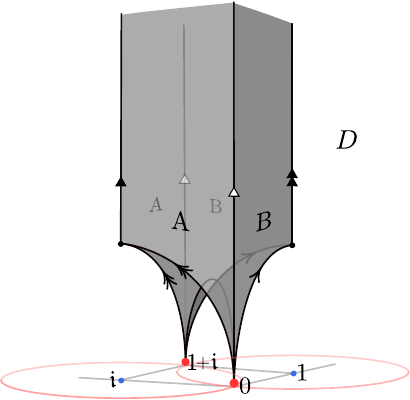}
   	\caption{
If an FAL complement covers an orbifold , with the (single) planar component covering a rigid cusp and all crossing circles covering flexible cusps, then the fundamental domain of the orbifold must be the domain $D$ above, with edge identifications as shown.}

 \label{fig:D_orb}
\end{figure}

\begin{figure}
	\begin{subfigure}{.27\textwidth}
 		\centering
   		\includegraphics[scale=1.5]{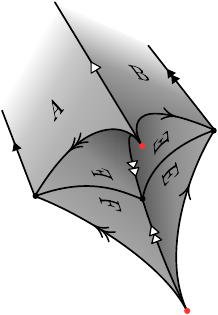}
   		\caption{}
   		\label{fig:orb2_gluing}
	\end{subfigure}
	\begin{subfigure}{.27\textwidth}
 		\centering
   		\includegraphics[scale=1.3]{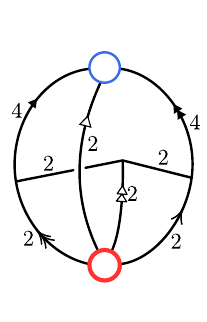}
   		\caption{}
   		\label{fig:orb2}
	\end{subfigure}	
	\begin{subfigure}{.4\textwidth}
 		\centering
   		\includegraphics[scale=1.6]{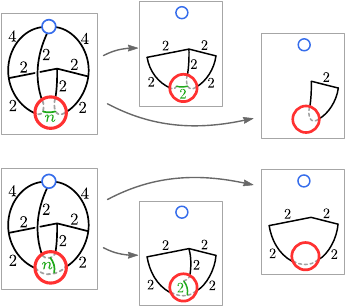}
   		\caption{}
   		\label{fig:cusp_kill}
	\end{subfigure}	

	\caption{There is exactly one possible gluing of $D$ that gives an orbifold with a flexible cusp covered by crossing circles, shown in (a). The resulting orbifold is shown in (b) as a graph in $\SS^3$, with three finite vertices and two ideal vertices, and edges labelled by their degree. If we fill the red cusp and apply the cusp-killing homomorphism as shown in (c), then a loop of 2-torsion remains (top and bottom shows two possible cases up to symmetry).}
	\label{fig:DL_orb}
\end{figure}

\begin{clm}\label{clm:eq_doms}
	$D_\orb = D$.
\end{clm}

\noindent\emph{Proof of \Cref{clm:eq_doms}.}
If $D_\orb\ne D$, then there exists an open half-space $S\subset\HH^3$ such that $D_\orb\subset D\cap S\neq \emptyset$. Let $\pi$ be the geodesic plane so that $S$ is a component of $\HH^3\setminus \pi$. If $\pi$ is a vertical plane then it must cut $D$, so the peripheral subgroup at $\infty$ must be strictly larger than $\Periph$. The only possible translation that does not map a blue horoball to a red horoball (up to composition with elements of $\Periph$) is $\tau_0=\mat 1 {1+\I} 0 1$. But $\sigma\tau_0$ is an order--4 rotation about the red horoball at $0$, so $\tau_0$ cannot be in $\pi_1(\orb)$ by our assumption that crossing circles do not cover $\SS^2(2,4,4)$ rigid cusps. Since there can be no order-4 rotation at a red horoball, the only possible rotation that does not exchange red and blue horoballs is an order-2 rotation at $\frac{1}{2}+\frac{\I}{2}$. If $\sigma_2$ is such a rotation, and $\sigma$ is the counter-clockwise rotation centered at $H_1$, then $\sigma_2\sigma$ is an order-4 rotation centered at the red horoball at $1+i$. Thus the rotation $\sigma_2$ cannot be in the peripheral subgroup, and so $\pi$ is not a vertical plane. Since the peripheral subgroup of $\pi_1(\orb)$ fixing $\infty$ is $\Periph$,  the covering map $\HH^3\to \orb$ is a homeomorphism on $\mathrm{int}(D\cap H_\infty)$. It follows that $\pi$ is a hemisphere of radius less than or equal to 1. Since $\orb$ has 2 cusps (one red and one blue), and since the two red vertices of $D$ are identified under $\Periph$, $\pi$ cannot cut off either of the red vertices of $D$. Since the two finite vertices of $D$ are in the (closed) horoball $H_\infty$, $\pi$ cannot cut either finite vertex off of $D$. It follows that all vertices of $D$ (finite and ideal) are contained in the closure of $\HH^3\setminus S$ (where we take the closure in $\HH^3\cup \del\HH^3$). Since $D$ is the convex hull of these vertices, $D$ must then be contained in the closure of $\HH^3\setminus S$. This contradicts our assumption that $S$ intersects $D$ non-trivially, and completes the proof of the claim.

\begin{clm}\label{clm:two_possibilties}
	 $\orb$ is the orbifold shown in \Cref{fig:orb2}. 

\end{clm}

\noindent\emph{Proof of \Cref{clm:two_possibilties}.}
Now consider the fundamental domain $D=D_\orb$. \Cref{fig:D_orb} shows edge identifications of this domain, induced by the action of $\Periph$ on $\HH^3$. Under this action, the four vertical faces of $D_\orb$ are identified in pairs, as shown. This leaves 2 triangular faces $\Delta_1$ and $\Delta_2$. The gluing of $\Delta_1$ must be realized by some $\phi\in\Isom^+(\HH^3)$ that either maps $\Delta_1$ to itself or to $\Delta_2$. The only two orientation-preserving maps satisfying this condition are the order--4 rotation about the edge connecting $0$ to $1+\I$, and the order--2 rotation whose axis lies on $S_\I$ and bisects $\Delta_1$. In the first case, the rotation is an order--4 rotation with fixed points on red horoballs, which is impossible given our assumption that crossing circles do not cover $\SS^2(2,4,4)$ rigid cusps in $\orb$. In the second case, $\Delta_1$ is folded onto itself along the axis of $\phi$, and similarly $\Delta_2$ is folded onto itself along a bisecting edge on $S_1$ (see \Cref{fig:orb2_gluing}, in which $\Delta_1$ is the union of the two faces labelled by $E$, and $\Delta_2$ is the union of faces labelled by $F$). With this gluing of $D$ we get the orbifold shown in \Cref{fig:orb2}. This completes the proof of the claim, and of the Lemma.
\end{proof}

For the following lemma we rely on the discussion of orbifold Dehn filling from \Cref{sub:orbDehnFilling}. This lemma will rule out the possibility that any Dehn filling of the orbifold in \Cref{fig:orb2} is covered by a knot complement. Note that the statement below does not require the knot complement to be $\edHT$, or even to be a filling of an FAL complement.

\begin{lem}\label{lem:no_knot_covers}
If $S^3 \setminus K$ a knot complement, then $S^3 \setminus K$ does not cover an orbifold obtained from filling the orbifold in \Cref{fig:orb2}. 
\end{lem}

\begin{proof}
Suppose $\SS^3\setminus K$ covers an orbifold $\orbQ$, and $\orbQ$ is obtained by Dehn-filling the flexible cusp of the orbifold $\orb$ in \Cref{fig:DL_orb}. By \cite[Proposition 2.3, Remark 2.4]{Hoffman2015small} and the proof of \cite[Corollary 4.11]{BBCW12}, the orbifold fundamental group of $\orbQ$ must be normally generated by the peripheral subgroup. Thus if we just quotient $\pi_1(\orbQ)$ by the normal subgroup generated by peripheral torsion, the resulting group should be torsion free (and the corresponding orbifold should have empty singular set). The two possible orbi-tangle fillings of $\orb$ (up to symmetry) are shown in the left frames of \Cref{fig:cusp_kill}. In this picture we draw the tangle strands as dotted lines to indicate that the drawing does not show the actual tangle, but rather just indicates which singularities on the cusp are joined to each other by the tangle strands (this will be sufficient for our purposes). The middle and right frames of \Cref{fig:cusp_kill} show the result of killing the peripheral torsion, with the middle frame corresponding to $n$ even and the right frame corresponding to $n$ odd. In both cases, we are still left with some order-2 torsion, so such an orbifold cannot be covered by a knot complement. 
\end{proof}

\begin{rem}
One can also prove a version of \Cref{lem:no_knot_covers} in the context of FAL fillings, by showing that any filling curve for the flexible cusp of $\orb$ that lifts to a $\frac{1}{q}$-filling of some crossing circle of the FAL, must lift to an integer filling in some other crossing circle. This follows from the fact that the horoballs $H_0$ and $H_{1+i}$ are identified via the order--4 rotation fixing $1$ and $\infty$, so the longitude of the crossing circle lifting to $H_0$ is identified with the meridian of the crossing circle lifting to $H_{1+i}$. Unfortunately this argument still allows for the $(1,1)$-filling on each cusp.  
\end{rem}

\subsection{Proof of \Cref{thm:ht_no_rigid_cusps}}\label{sec:rigid}

We are now ready to prove \Cref{thm:ht_no_rigid_cusps}. 

\begin{proof}[Proof of \Cref{thm:ht_no_rigid_cusps}]

Suppose $\SmK$ is an $\edHT$-filling of $\SmL$ and suppose $\SmK$ covers a rigid cusped orbifold $\orb$. By \Cref{prop:HT_non-empty}, the cover $\SmK \rightarrow \orb$ induces a cover $\SmL \rightarrow \orbQ$, where $\orbQ \cong \orb \setminus \sqcup_{i=1}^{n} p(\gamma_{i})$. By construction, $\orbQ$ must have a rigid cusp, coming from the single rigid cusp of $\orb$, which the cusp corresponding to $K_0$ covers. At the same time, $\orbQ$ must have non-rigid cusps that are tori coming from drilling out the set of geodesics $\{p(\gamma_{i})\}_{i=1}^{n}$ from $\orb$. Consequently, either every crossing circle of $\SS^3\setminus L$ covers a non-rigid cusp in $\orbQ$, contradicting either \Cref{lem:rigid} or \Cref{lem:no_ord3}, or $\orbQ$ is the orbifold in \Cref{fig:DL_orb}, contradicting  \Cref{lem:no_knot_covers}.
\end{proof}

\section{Symmetries of $\edHT$ and generic knots}\label{sec:sym}

The symmetries of an $\edHT$ knot complement are restricted in the sense that the $\epsilon$-thin parts of the manifold must be permuted under any symmetry. When a knot complement is $\edHT$ and generic such a permutation is trivial, and we can narrowly characterize the symmetry group of such a knot complement. In what follows, let $\K=\SS^3\setminus K $ be an $\edHT$ and generic knot complement obtained by Dehn filling all of the crossing circles of an FAL complement $\L = \SS^3 \setminus L$, where $L$ has a single planar component. Our goal here is to prove the following theorem on the structure of the (orientation preserving) symmetry group $Sym^+(\K)$.

\begin{thm}
\label{thm:HTHDcovering}
 If $\K$ is an $\edHT$ and generic knot complement, then its orientation preserving symmetry group satisfies ${|Sym^+(M)| \leq 4}$.
\end{thm}

In fact, we will prove something slightly stronger, namely, that $|Sym^+(\K)| = 1$, $2$, or $4$. In the case that the group is order two, the non-trivial element could be a strong involution fixing points on the cusp or an order 2 symmetry acting freely on the cusp, but fixing points in the interior of $\K$. We collect these facts in \Cref{cor:free_cusp_symms} at the end of this section.

Together, \Cref{thm:ht_no_rigid_cusps} and \Cref{thm:HTHDcovering} provide important information about the commensurability class of $\K = \mathbb{H}^{3} / \Gamma$. Since $\edHT$ knot complements must be non-arithmetic (see \Cref{subsec:SuffTwisted}), work of Margulis \cite{Mar91} then implies that there exists a unique minimal (orientable) orbifold in the commensurability class of $M$, namely, $\orb = \mathbb{H}^{3} / C^{+}(\Gamma)$, where $C^+(\Gamma) = \{ g \in \Isom^+(\mathbb{H}^{3})  :  | \Gamma : \Gamma \cap g \Gamma g^{-1} | < \infty \}$. Our work places strong restrictions on  the covering map $\psi: \K \rightarrow \orb$ and the geometric structure of $\orb$.

\begin{cor}\label{cor:HTHDwithRigidCusps}
If $\K$ is an $\edHT$ and generic knot complement, then the degree of the minimal orbifold covering map $\psi:\K \rightarrow \orb$ is at most $4$. 
\end{cor}

\begin{proof}
 Hidden symmetries of $\K$ correspond with elements of $C^{+}(\Gamma) \setminus \L^{+}(\Gamma)$, where $\L^{+}(\Gamma)$ is the normalizer of $\Gamma$ in $\Isom^{+}(\mathbb{H}^{3})$.  \Cref{thm:ht_no_rigid_cusps} and \Cref{thm:NR} imply that $\K$ has no hidden symmetries, and so, $C^+(\Gamma) = \L^+(\Gamma)$. Since $Sym^+(\K) \cong \L^+(\Gamma) / \Gamma$, we have that $\orb = \mathbb{H}^{3} / C^{+}(\Gamma) = \mathbb{H}^{3} / \L^{+}(\Gamma) = \K / Sym^{+}(\K)$. Thus,  \Cref{thm:HTHDcovering} restricts the degree of $\psi$ to at most $4$. 
\end{proof}

\vspace{0.1in}

As noted in the proof above, $C^{+}(\Gamma) = \L^+(\Gamma)$, and so $\psi: \K \rightarrow \orb$ is just the quotient of $\K$ via $Sym^+(\K)$. This fact only relies on \Cref{thm:ht_no_rigid_cusps}, and we will assume that it holds going forward. Thus, proving \Cref{thm:HTHDcovering} is equivalent to proving that the covering map  $\psi: \K \rightarrow \orb$ is at most degree $4$.

Our work partially follows the same line of argument as the proof of \cite[Theorem 1.7]{FM17}, which shows that certain hyperbolic $3$-manifolds are the minimal orbifold in their respective commensurability classes. Their work requires some very strong conditions placed on the geometry of thrice-punctured spheres inside of a cusped hyperbolic $3$-manifold in order to obtain the desired result. In a similar fashion, we exploit the geometry of thrice-punctured spheres in our FAL complements to help reach our desired covering restriction. 

Let $\orb \cong M / Sym^{+}(M)$. \Cref{prop:HT_non-empty} implies that the quotient map $\psi:\K \rightarrow \orb$  induces a quotient map  $\psi_{\L}: \L \rightarrow  \orbQ$ with $\orbQ \cong \orb \setminus \sqcup_{i=1}^{n} \psi(\gamma_{i})$, where $\gamma_{i}$ is the geodesic core of the $i^{th}$ surgery solid torus. Furthermore, \Cref{prop:symrelation} tells us that $\orbQ$ is the quotient of $N$ by the subgroup of $Sym^{+}(N)$ that maps each cusp of $N$ to itself. We will now consider how thrice-punctured spheres in $N$ behave under the quotient map $\psi_{N}$. 

For our hyperbolic FAL complement $\L = \mathbb{S}^{3} \setminus L$, consider a crossing disk, as described in \Cref{subsec:PD}. Such a crossing disk is isotopic in $\L$ to a totally geodesic thrice-punctured sphere $S \subset \L$ \cite[Lemma 2.1]{Pur11}. Let $p_{1}$, $p_{2}$, and $p_{3}$ be the three punctures of $S$. Let $\H$ be the preferred horoball packing for $L$, with associated covering map $\pi:\HH^3\to \L$, and let $n(K_0),n(C_1),\dots,n(C_n)$ be neighborhoods of the cusps of $\L$ that lift to $\H$. The neighborhood $n(K_0)$ contains neighborhoods of two of the punctures of $S$, say $p_2$ and $p_3$, given by the components of $S\cap n(K_0)$. Let $D_j$ be the connected component of $S\cap n(K_0)$ containing $p_j$, $j=2,3$. The third puncture $p_1$ is contained in $n(C_{i_S})$ for some crossing circle $C_{i_S}$, and we denote by $D_1$ the neighborhood of $p_1$ given by $S\cap n(C_{i_S})$.

Let $\psi_S=\psi_N|_S$ be the restriction of $\psi_{\L}$ to $S$. The next lemma restricts the degree of $\psi_S$, and is the first step toward restricting the degree of $\psi$.

\begin{lem}
\label{lemma:cover3PS}
Let $S \subset \L$ be a crossing disk. Then $\psi_S$ is either a 1-to-1 map, a 2-to-1 map or a 4-to-1 map. 
\end{lem}

\begin{proof}
Recall that $\psi_{\L}$ is the quotient map $\L/G$, where $G=f(Sym^+(\K))\le Sym^+(\L)$, and $S\subset \L$ is totally geodesic. It follows that $\psi_S$ is a quotient map for the quotient $S/G_S$, where $G_S\le G$ are the symmetries in $G$ that restrict to symmetries of $S$. By \Cref{prop:symrelation}, $G$ must map each cusp of $\L$ to itself. Thus, any symmetry in $G_S$ must take $D_1$ to itself, and permute $D_2$ and $D_3$ (possibly trivially). \Cref{fig:thrice_punctured_symms} shows three non-trivial symmetries on $S$ that satisfy these conditions, and generate a group of order 4.

We claim that the symmetries described in \Cref{fig:thrice_punctured_symms} are the only symmetries of $S$ satisfying the required conditions. First, a thrice-punctured sphere $S$ has a unique hyperbolic structure, which implies that every homeomorphism of $S$ is isotopic to an isometry. Since homeomorphisms of $S$ are completely determined by their action on the cusps, we only need to consider how an isometry could permute cusps and possibly switch orientations on cusps of $S$. For us, the cusp corresponding to $p_{1}$ must be fixed under any isometry under consideration, and so, we only have 4 possible isometries coming from permuting the two remaining cusps and possibly switching orientations. 
\end{proof}

\begin{figure}
 	\centering
   	\includegraphics[scale=2]{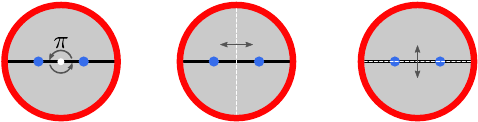}
   	\caption{
A schematic of the three non-trivial symmetries of the thrice punctured sphere that result in two cusped quotients. Using the standard fundamental domain for the thrice punctured sphere, two ideal triangles with vertices at $0$,$-1$,$1$, and $\infty$, these symmetries can be realized by a rotation of order 2 fixing $\I$, a reflection through $0$,$\infty$ and a reflection through the line with endpoints $-1$ and $1$.}
 \label{fig:thrice_punctured_symms}
\end{figure}

Let $R = \psi_{\L}(S)$. In order to leverage the degree of $\psi_S$ determined above to obtain information about the degree of $\psi$, we will need to determine the number of pre-images in $\psi_{\L}^{-1}(R)$.  To this end, we first show that $R$ must be an embedded, totally geodesic $2$-orbifold in $\orbQ$. The following lemma implies this fact.

\begin{lem}
\label{lemma:3PSintersections}
The quotient $\psi_{\L}(S)$ has no transverse self-intersections.
\end{lem}

\begin{proof}

Here, we follow a similar procedure to the one used in  Lemma 2.8 of Futer--Millichap \cite{FM17}. We first rule out transverse self-intersections between the images of the 2-dimensional cusp neighborhoods, $\psi(D_{i})$ for $i= 1, 2, 3$. Then we show this implies that we can not have transverse self-intersections anywhere else in $\psi_{\L}(S)$, completing the proof.  

As noted in the proof of \cite[Lemma 2.8]{FM17}, the only way that $\psi_{\L}(D_i)$ can have transverse intersection with $\psi_{\L}(D_j)$ is if the cusp neighborhoods containing $D_i$ and $D_j$ are identified under $\psi_{\L}$, and $\psi_{\L}(\del D_i)$ and $\psi_{\L}(\del D_j)$ represent distinct slopes. The first condition can only be met for $D_2$ and $D_3$ (since $\psi_{\L}$ does not exchange cusps of $\L$), and the boundary of each of these meets the boundary torus of $n(K_0)$ in a meridian.

Since the cusps in $\orbQ$ are non-rigid, this meridian maps down to a well-defined slope in $\orbQ$. So, the slopes of $\psi_{\L}(\del D_i)$ and $\psi_{\L}(\del D_j)$ must be the same in $\orbQ$. Thus, transverse self-intersections do not occur between the cusps of $\psi_{\L}(S)$.

It remains to show that $\psi_{\L}(S)$ does not have transverse self-intersections outside of its cusp neighborhoods. For this the exact same argument as the one used in the proof of \cite[Lemma 2.8]{FM17} goes through unchanged, and we refer the reader to this work for explicit details. In short, if $\psi_{\L}(S)$ has a transverse self-intersection somewhere outside of its cusp neighborhoods, then there exists a path of transverse self-intersections in $\psi_{\L}(S)$ that goes into some cusp of $\psi_{\L}(S)$,  contradicting the previous paragraph. Therefore, $\psi_{\L}(S)$ has no transverse self-intersections. 
\end{proof}

Next we show that for some crossing disk $S$, the embedded totally geodesic 2-orbifold $R=\psi_{\L}(S)$ has pre-image exactly $S$. To do this, we will first need to show that there is a crossing circle that has only one generalized crossing disk, which is its standard crossing disk (recall  \Cref{subsec:PD} for the definition of generalized crossing disks). We first need a technical lemma that relies on the following definition. For $\SmL$, we have that $P_1$ and $P_2$ are the two polyhedra with shaded faces as in \Cref{sec:background}. 
The \emph{nerve} of $\SmL$ (or just the nerve of $L$) is the dual 1-skeleton of the unshaded faces of $P_1$ (see also \cite{Pur11} for more background).

\begin{lem}\label{lem:subgraph}
Let $C$ be a crossing circle of an FAL $L$. If $C$ has more than one generalized crossing disk, then the nerve $\nerve$ of $L$ contains the subgraph $\nerve_C$ shown in \Cref{fig:subgraph}, in which the red edge corresponds to the crossing circle $C$.	
\end{lem}
\begin{proof} Let $P_i$, $i\in\{1,2\}$, be the polyhedra coming from the polyhedral decomposition of $L$, and let $\nerve$ be the nerve. If $C$ has a generalized crossing disk that is not the standard crossing disk, then there are triangles $\tau_i\subset P_i$ such that $\tau_i$ has $C$ as a vertex, and $\tau_i$ separates $P_i$. On $P_i$, $\tau_i$ is a path across white faces connecting 3 distinct vertices. Such a path gives a 3-cycle in the nerve, one edge of which must correspond to $C$ since $\tau_i$ has a vertex corresponding to $C$. The two triangles of $\nerve$ adjacent to $C$ correspond to the standard crossing disk of $C$, so there must be 2 additional edges that form the 3-cycle corresponding to $\tau_i$.	
\end{proof}

\begin{prop}
\label{prop:CDs}
Let $L$ be an FAL. Then $L$ has at least one crossing circle having only one generalized crossing disk.	
\end{prop}

\begin{proof}
Let $\nerve$ be the nerve for $L$. Suppose every crossing circle has $\ge 2$ generalized crossing disks. The preceding \Cref{lem:subgraph} then implies that for every crossing circle 	there is a subgraph as shown in \Cref{fig:subgraph}. In particular, for every crossing circle $C_i$ there is a 3-cycle $\beta_i$ in $\nerve$ that is non-trivial in the sense that it does not bound a triangle. Identify $\SS^2$ with $\widehat{\RR^2}=\RR^2\cup\{\infty\}$, so that the point $\infty$ is disjoint from the edges and vertices of $\nerve$. Let $\Delta_i$ be the component of $\widehat{\RR^2}\setminus \beta_i$ that does not contain $\infty$. For any other crossing circle $C_j$, the corresponding non-trivial 3-cycle $\beta_j$ is either disjoint from $\Delta_i$ (though it may intersect $\beta_i$), or it has 1 or more vertices contained in $\Delta_i$, in which case it is contained in $\Delta_i\cup \beta_i$. Note that $\Delta_i$ does not contain the vertices of $\beta_i$ as we have defined it. It follows that we can find a crossing circle $C_{i_0}$ whose non-trivial 3-cycle $\beta_{i_0}$ is innermost, in  the following sense: the component $\Delta_{i_0}$ of $\widehat{\RR^2}\setminus \beta_{i_0}$ that does not contain $\infty$, contains no vertex of any other non-trivial 3-cycle associated to a crossing circle. 

Now consider the subgraph $\nerve_{C_{i_0}}$ given by \Cref{lem:subgraph}. The (arrow shaped) quadrilateral $Q\subset \nerve_{C_{i_0}}$ must contain a red colored edge (i.e., one associated to a crossing circle), otherwise $Q$ would subdivide into triangles, each having all three edges uncolored, contradicting the fact that every triangle of a nerve must have exactly one colored edge. This colored edge cannot be part of a non-trivial 3-cycle, as such a 3-cycle would have a vertex in $\Delta_{i_0}$, contradicting our requirement that $\beta_{i_0}$ is innermost. Thus the associated crossing circle has no generalized crossing disk other than the standard one, contradicting our assumption. \end{proof}

\begin{figure}
 	\centering
   	\includegraphics[scale=1]{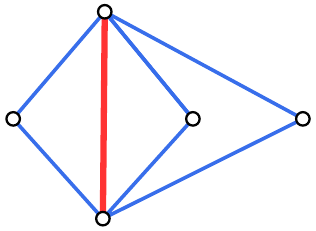}
   	\caption{The subgraph $\nerve_C$, with a (vertical) red edge corresponding to the crossing circle $C$, must be present if $C$ has more than one generalized crossing disk.}
   	\label{fig:subgraph}
\end{figure}

We are now ready to prove our main theorem for this section. 

\begin{proof}[Proof of \Cref{thm:HTHDcovering}]
Let $S$ be a totally geodesic thrice-punctured sphere that is a crossing disk for a crossing circle $C_{i}$ of $L$. First, we determine the structure of $\psi_{\L}^{-1}(R) = \psi_{\L}^{-1} \circ \psi_{\L} (S)$.  \Cref{lemma:3PSintersections} implies that $\psi_{\L}(S)$ is an embedded totally geodesic cusped $2$-orbifold in $\orbQ$. Thus, $\psi_{\L}^{-1}(R)$ must be a disjoint union of embedded totally geodesic cusped surfaces in $\L$. Let $F$ be a component of $\psi_{\L}^{-1}(R)$.

First, we will show that $F$ must be a totally geodesic thrice-punctured sphere in $\L$. This follows from the same argument used in the proof of \cite[Theorem 2.3]{FM17} and requires analyzing a particular decomposition of $S$ (coming from a Ford domain for $S$), which we now describe. Consider our preferred horoball packing $\H$ coming from \Cref{thm:hbp_FP07}. Under our covering $\pi:\HH^3\to \L$, this horoball packing descends to a particular choice of cusp neighborhoods for $\L$, which then (geometrically) determine the set of $2$-dimensional cusp neighborhoods $\{ D_i \}$ of $S$. Let $E_{i} \subset S$ be the closure of the set of points in $S$ that are closer to $D_{i}$ than to any other $D_{j}$. For this cusp expansion, the set $\{D_{i}\}$ is invariant under the symmetries of $S$ and each $\partial D_{i}$ has length $2$. Thus, the elements of $\{E_{i}\}$ are all isometric to a once-punctured disk with non-ideal boundary consisting of two geodesic segments meeting at angles of $2\pi/3$. We call this decomposition of $S$ into $\{E_{i}\}$ its Ford decomposition and each $E_{i}$ a Ford cell of $S$.

We now bootstrap off of this Ford decomposition of $S$ to determine the corresponding Ford decomposition of $F$. The Ford decomposition of $F$ (relative to our preferred horoball packing $\H$) decomposes $F$ into a collection of Ford cells, one for each cusp of $F$. Since each $\psi_{\L} (\partial D_{j})$ realizes exactly one slope on each cusp of $\orbQ$ (see the proof of \Cref{lemma:3PSintersections}), each cusp of $F$ has the same slope as some $\partial D_{j}$. Thus, each cusp of $F$ has a boundary slope of length $2$, which implies each cusp of $F$ has area $2$ and is isometric to a copy of $D_i$.  Since the density of the cusp neighborhoods of $F$ is the same as that of $S$, each Ford cell of $F$ must be isometric to $E_{i}$. The only complete, connected hyperbolic surface that can be constructed by gluing together some number of copies of $E_{i}$ is a thrice-punctured sphere, as needed. 
 
 We now claim  that $F$ must be a generalized crossing disk. \Cref{prop:symrelation} implies that there is a one-to-one correspondence between cusps of $\L$ and cusps of $\orbQ$, and so, we know that $F$ must have one cusp coming from intersecting $n(C_{i_S})$ and two cusps coming from intersecting $n(K_{0})$. The classification of thrice-punctured spheres in FAL complements given by Morgan--Spyropoulus--Trapp--Ziegler (\cite[Propositions 9 and 13]{morganbelted}) implies that $F$ must be a generalized crossing disk.

Since we are quotienting $\L$ by orientation-preserving symmetries, the singular locus of $\orbQ$ is at most one-dimensional, and so, we can conclude that the degree $d(\psi_{\L})$ of $\psi_{\L}$ is the same as the degree of the restriction $\psi_{\L}|_{\psi_{\L}^{-1}(R)}$ of $\psi_{\L}$ to the pre-image of $R=\psi_{\L}(S)$ for any crossing disk $S$. \Cref{lemma:cover3PS} tells us that $\psi_S$ is at most a 4-to-1 covering map. In addition,  \Cref{prop:CDs} along with the previous paragraph implies that there exists some crossing disk $S_0$ such that $\psi_{\L}^{-1}(R_0) = \psi^{-1} \circ \psi (S_0)  = S_0$.  So, for this particular crossing disk, we have $d(\psi_{\L}|_{\psi_{\L}^{-1}(R_0)}) \le 4$. Therefore, $d(\psi_{\L})=d(\psi)$ is at most $4$, which implies that $|Sym^+(\K)| \leq 4$.   \end{proof}

\begin{rem}\label{rem:tight}
We observe that this bound on the number of symmetries is tight: there exists infinitely many $\edHT$ and generic knot complements that have an order four (orientation-preserving) symmetry group. For instance, any $\edHT$ and generic 2-bridge knot complement will have exactly an order four (orientation-preserving) symmetry group;  of course this is well-known (see \cite{sakuma1990geometries} for example).
\end{rem}

We can completely describe the symmetry group of $\K$ based on how its inclusion into $Sym^+(N)$ acts on thrice punctured spheres. Thus, the proof of the following corollary follows directly from extending the action of the symmetries exhibited in \Cref{fig:thrice_punctured_symms}  to $N$ and then $\K$. 

\begin{cor}\label{cor:free_cusp_symms}
Let $\K$ be a $\edHT$ and generic knot complement. Then $Sym^+(\K)$ has at most three non-trivial elements. Furthermore, the group has zero or one element that acts freely on the cusp. 
\end{cor}

\section{Commensurability classes of $\edHT$ and generic knots}\label{sect:unique_in_comm_class}

In this section, we show that an $\edHT$ and generic knot complement with at least 9 twist regions, each having at least 6 crossings, is the only knot complement in its commensurability class. Proving this result will rely on our classifications of hidden symmetries and symmetries from \Cref{sec:hidden_syms} and  \Cref{sec:sym}, along with an analysis of short filling slopes for these knot complements and the orbifolds that they cover. In particular, the aforementioned lower bound on the number of twists regions and crossings will imply a lower bound on filling slope lengths, allowing us to obstruct non-hyperbolic surgeries, which by a theorem of Boileau--Boyer--Cebanu--Walsh will imply the claimed result. 

To be more specific, Boileau, Boyer, Cebanu and Walsh \cite{BBCW12} showed that if a hyperbolic knot complement $\SS^3 \setminus K$ without hidden symmetries is commensurable with a second (distinct) knot complement $\SS^3 \setminus K'$, then both knot complements cover an orbifold $\orbQ$ with a torus cusp, and both covers are cyclic. In fact, the authors show that the set of knot complements that cover $\orbQ$ are in bijective correspondence with the set of finite cyclic fillings of $\orbQ$ (see \cite[Proposition 4.13]{BBCW12}). Their work also defines and characterizes \define{orbi-lens spaces}, which are cyclic quotients of $\SS^3$ (in the case that the quotient acts freely on $\SS^3$, the resulting space is just a lens space). Thus if $\SS^3\setminus K$ is commensurable with another knot complement, then it covers an orbifold with two distinct orbi-lens space fillings. The main technical result of this section will rule out the existence of orbi-lens space fillings on the quotients of $\edHT$ and generic knot complements.

\subsection{Geometric Bases}
\label{sec:GeometricBases}

We begin by discussing geometric bases for co-compact groups of translations in the plane. We will identify the plane with $\CC$, with the structure of a 2-dimensional $\RR$-vector space. We say a  \define{geometric basis} for such a group $G$ is a set of two $\RR$-linearly independent elements $\{a,b\}$ of $\CC$ such that for any $n,m \in \ZZ$ we have $|a| \leq |na+mb|$ and $|b| \leq |na+mb|$. 

A geometric basis can be found by a greedy algorithm. We start with the observation that given two non-zero elements $a,b \in \CC$, $|ta+b|$ has a unique minimum for  $t\in \RR$ and one or two minima for $t \in \ZZ$. Moreover, if $|ta+b| \le |(t+1)a+b|$, then either $|(t-1)a+b| < |ta+b|$ or $|ta+b|$ is a minimum for $t \in \ZZ$. (Of course if $|(t \pm 1)a+b| = |ta+b|$, then both $|(t\pm 1)a+b|$ and $|ta+b|$ are minima.) In this way, we can start with two elements that form a basis $\{a,b\}$ (assume $|a| \leq |b|$) and replace $b$ with $ta+b$ to form a basis with shorter entries. If  $|ta+b| < |a|$ we switch the roles of $a, b$ above and repeat until the process terminates. Note that even up to ordering such that $|a| \leq |b|$ a geometric basis might not be unique. For example, non-uniqueness occurs if  $|b|= |a+b|$ or $|b|=|-a+b|$.

The following proposition relates a geometric basis for a group of translations $G_1$ to geometric bases of an index 2 subgroup $G_2$. When used in context, $G_1$ will correspond to the peripheral subgroup of an orbifold and $G_2$ will be the peripheral subgroup of its manifold cover. Of course, the proposition applies much more generally.

\begin{prop}\label{prop:quotientbases}
Let $G_1$ be co-compact group of translations with geometric basis $\{a, b\}$, $|a|\le |b|$, and let $G_2$ be an index 2 subgroup. Then one of the following holds:
\begin{enumerate}
\item $a  \in G_2$ and $2b \in G_2$, and $\{a, 2b\}$ or $\{a, a \pm 2b\}$ form a geometric basis for $G_2$,
\item $2a  \in G_2$ and $b \in G_2$ and $\{b, 2a\}$ or $\{b, b \pm 2a\}$ form a geometric basis for $G_2$,
\item $a\pm b  \in G_2$ and a geometric basis is a subset of $\{a \pm b, 2a\}$.   
 \end{enumerate}
  
\end{prop}

The above proposition follows from a straightforward application of the greedy algorithm to each of the  kernels of the three possible surjective maps from $\ZZ \times \ZZ \rightarrow \ZZ/2\ZZ$. 

\begin{lem}\label{lem:bounds}
Let $G_1$ be co-compact group of translations with geometric basis $\{a_1, b_1\}$, and let $G_2$ be an index 2 subgroup with geometric basis $\{a_2,b_2\}$. If $|a_2|=2$ and $|b_2|>16$, then either $|b_1|>6$ or $|a_1|>6$.
	
\end{lem}

\begin{proof}
If $|a_1|\le |b_1|$, then \Cref{prop:quotientbases} gives the possible bases for $G_2$. If $|b_1|<|a_1|$ then cases (1) and (2) in \Cref{prop:quotientbases} are unaffected, and in case (3) we simply exchange the roles of $a=a_1$ and $b=b_1$, giving the additional possible bases $\{2b_1, a_1\pm b_1\}$. \Cref{tab:short_quotient_slopes} shows all possible bases $\{a_2,b_2\}$, and the resulting bounds on $a_1$ and $b_1$. 
\end{proof}

\begin{table}[ht]
\begin{tabular}{|c|c|c|c|c||c|c|c|c|}
\hline
&$a_2$  & $b_2$ &  $|a_2|$ & $|b_2|$  & $a_1$ & $b_1$ & $|a_1|$ & $|b_1|$\\ 
\hline 
\hline
\multirow{3}{*}{(1)} &$a_1$ & $2b_1$ & 2 & $>12$ & $a_2$ & $ \frac{1}{2}b_2$ & 2 & $>6$\\ 
&$a_1$ & $a_1+2b_1$  & $2$ & $>14$ & $a_2$ & $\frac{1}{2}(b_2-a_2)$ & $2$ & $>6$\\ 
&$a_1$ & $a_1-2b_1$ & $2$ & $>14$  & $a_2$ & $\frac{1}{2}(a_2-b_2)$ & $2$ & $>6$\\
\hline
\multirow{3}{*}{(2)}& $2a_1$ & $b_1$ & $2$ & $>6$  & $\frac{1}{2}a_2$ & $ b_2$ & $1$ & $>6$\\ 
&$b_1+2a_1$ & $b_1$ & $2$ & $>6$ & $\frac{1}{2}(a_2-b_2)$ & $ b_2$ & $>2$ & $>6$\\ 
&$b_1-2a_1$ & $b_1$ & $2$ & $>6$  & $\frac{1}{2}(b_2-a_2)$ & $b_2$ & $>2$ & $>6$\\ 
\hline
\multirow{10}{*}{(3)} &\cellcolor{Gray} $a_1+b_1$ &\cellcolor{Gray} $a_1-b_1$ &\cellcolor{Gray} $2$ &\cellcolor{Gray} $>14$ &\cellcolor{Gray}  $\frac{1}{2}(a_2+b_2)$ &\cellcolor{Gray} $\frac{1}{2}(a_2-b_2)$ &\cellcolor{Gray} $>6$ &\cellcolor{Gray} $>6 $ \\
&\cellcolor{Gray} $a_1-b_1$ &\cellcolor{Gray} $a_1+b_1$ &\cellcolor{Gray} $2$ &\cellcolor{Gray} $>14$ &\cellcolor{Gray} $\frac{1}{2}(a_2-b_2)$ &\cellcolor{Gray} $\frac{1}{2}(a_2+b_2)$ &\cellcolor{Gray} $>6$ &\cellcolor{Gray} $>6$ \\
&\cellcolor{Gray} $a_1+b_1$ &\cellcolor{Gray} $2a_1$ &\cellcolor{Gray} $2$ &\cellcolor{Gray} $>16$ &\cellcolor{Gray} $\frac{1}{2}b_2$ &\cellcolor{Gray} $a_2- \frac{1}{2}b_2$ &\cellcolor{Gray} $>8$ &\cellcolor{Gray} $>6$\\
& $2a_1$ & $a_1+b_1$  & $2$ & $>7$ &$\frac{1}{2}a_2$ & $b_2- \frac{1}{2}a_2$ & $1$ & $>6$\\
&\cellcolor{Gray} $a_1+b_1$ &\cellcolor{Gray} $2b_1$ &\cellcolor{Gray} $2$ &\cellcolor{Gray} $>16$ &\cellcolor{Gray} $a_2-\frac{1}{2}b_2$ &\cellcolor{Gray} $\frac{1}{2}b_2$ &\cellcolor{Gray} $>6$ &\cellcolor{Gray} $>8$\\
& $2b_1$ & $a_1+b_1$  & $2$ & $>7$ & $b_2-\frac{1}{2}a_2$ &$\frac{1}{2}a_2$ & $ >6$ & $1$\\
&\cellcolor{Gray} $a_1-b_1$ &\cellcolor{Gray} $2a_1$ &\cellcolor{Gray} $2$ &\cellcolor{Gray} $>16$ &\cellcolor{Gray} $a_2+\frac{1}{2}b_2$ &\cellcolor{Gray} $\frac{1}{2}b_2$&\cellcolor{Gray} $>6$ &\cellcolor{Gray} $>8$\\
& $2a_1$ & $a_1-b_1$  & $2$ & $>7$ &$\frac{1}{2}a_2$&  $b_2+\frac{1}{2}a_2$ &$1$ & $>6$\\
&$a_1-b_1$ & $2b_1$ & $2$ & $>16$ & $a_2+\frac{1}{2}b_2$ & $\frac{1}{2}b_2$& $>6$ & $>8$\\
& $2b_1$ & $a_1-b_1$  & $2$ & $>7$ & $b_2+\frac{1}{2}a_2$ & $\frac{1}{2}a_2$& $>6$ & $1$\\
\hline

\end{tabular}

\caption{\label{tab:short_quotient_slopes}
The relations and lengths for $a_1$, $b_1$, $a_2$, $b_2$, with rows grouped by the case breakdown in \Cref{prop:quotientbases}. A shaded row means that the choices of $a_1$ and $b_1$ cannot form a geometric basis given the constraints on $a_2$ and $b_2$, but are still relevant when accounting for short parabolic elements. For example, if $a_1=\frac{1}{2}(a_2+b_2)$  and $b_1=\frac{1}{2}(a_2-b_2)$ are a basis with $|a_2|=2$ and $|b_2|> 13$ then  $a_2 = a_1+ b_1$ is shorter than both basis elements. Finally, we note that $|b_2|>16$ is the cutoff we need.
 }
\end{table}

\subsection{Orbifold Dehn fillings}
\label{sect:6Theorem}

We say a \define{good orbifold} is an orbifold that is covered by a manifold. Otherwise, it is \define{bad}. An orientable 2-orbifold is \define{bad} if it is $S^2(n)$ or $S^2(n,m)$ such that $n\ne m$ (see for example \cite[Proposition 2.10]{BMP03}). Rather directly, we can see that an orientable $3$-orbifold is bad if it contains a bad $2$-suborbifold. The following theorem of Boileau, Mallot, and Porti shows this is actually equivalent to being bad:

\begin{thm} [{\cite[Corollary 3.28]{BMP03}}]  A compact 3-orbifold is the quotient of a compact 3-manifold by an orientation preserving finite group action if and only if it does not contain a bad 2-suborbifold. \label{thm:BMP03}
\end{thm}

We now prove a weak corollary of the 6 Theorem of Agol \cite{agol2000bounds} and Lackenby \cite{lackenby2000word}, which will be sufficient for our purposes. In particular, our eventual goal is to obstruct orbi-lens space fillings of certain  hyperbolic orbifolds having no short filling slopes (with length as defined in \Cref{sub:orbDehnFilling}). Since orbi-lens spaces are good {(they are covered by $\SS^3$)}, it is enough then to show that filling along a long slope (i.e., one of length greater than 6) results in an orbifold that is either hyperbolic or contains a bad 2-suborbifold. Our statement is as follows:

\begin{thm}
\label{thm:orb_6_theorem}
Let $\orbQ$ be an (orientable) hyperbolic 3-orbifold  that is the  quotient of a hyperbolic $3$-manifold with torus cusps and assume  the singular locus of $\orbQ$ is either empty, a single embedded knot $\Sigma$ consisting of cone points of order $n$, or a set of properly embedded arcs and simple closed curves all having cone points of order $2$. Furthermore, denote by $\mathcal{H}_\orbQ$ an (embedded) horoball packing of $\orbQ$. Let $\alpha$ be a multi-slope of $\orbQ$ such that the length of each component of $\alpha$ measured by its displacement in $\mathcal{H}_\orbQ$ is greater than $6$. Then $\orbQ(\alpha)$, the orbifold resulting from filling along $\alpha$, is either hyperbolic or contains a bad 2-suborbifold. 
\end{thm}

\begin{proof}
If $\orbQ$ is a manifold, then the standard 6 Theorem applies. So, assume the set of cone points of $\orbQ$ is non-empty.

If $\orbQ(\alpha)$ contains a bad 2-suborbifold, we are done.  Thus we may assume that $\orbQ(\alpha)$ is good and appeal to \Cref{thm:BMP03}, which gives a manifold cover $p:M_\alpha \rightarrow \orbQ(\alpha)$, for some 3-manifold $M_\alpha$. Let $\overline{\orbQ}$ be the result of drilling out from $\orbQ(\alpha)$ the surgery solid tori and surgery solid pillowcases (so that $\mathrm{int}(\overline{\orbQ})\cong \orbQ$), and let $\overline{M}=p^{-1}(\overline{\orbQ})$. The covering map $p$ restricts to a cover $p|_{\overline{M}}:\overline{M} \to \overline{\orbQ}$, which then extends to a cover $p_\orbQ:\mathrm{int}(\overline{M})\to \orbQ$. Note that $\mathrm{int}(\overline{M})$ is a cusped hyperbolic 3-manifold which admits a horoball packing consistent with $\mathcal{H}_\orbQ$. Under this packing each of the lifts of curves in $\alpha$ have length greater than $6$. Thus, $M_\alpha$ is hyperbolic by the standard 6 Theorem and $\orbQ(\alpha)$ is hyperbolic as well. 
\end{proof}

We now prove the main theorem of this section. We define an \define{exceptional} filling of a hyperbolic orbifold to be a filling that results in a good non-hyperbolic orbifold. We separate out the case that a filling results in  bad orbifold, because this distinction is relevant to our argument. This theorem was stated in the introduction. We restate here for convenience.

\begin{thmn}[\ref{lem:second_slope_long}]
\secondSlopeLongText
\end{thmn}

\begin{proof}
Futer and Purcell's main result from \cite[Corollary 1.8]{FP07} establishes that $\SS^3 \setminus K$ does not admit any non-trivial exceptional surgeries.  

By \Cref{cor:free_cusp_symms},
$\SS^3 \setminus K$ has at most one non-trivial symmetry $\beta$ that acts freely on the cusp.  Assume $\orbQ$ is a non-trivial orbifold quotient of $\SS^3 \setminus K$ by such a symmetry $\beta$, and let $\orbQ_L$ be the orbifold quotient of $\SS^3 \setminus L$ by the corresponding symmetry $f(\beta)$ of $\SS^3\setminus L$ given by the \Cref{prop:symrelation}. 

By \Cref{prop:symrelation}, $f(\beta)$ maps each cusp of $\SS^3\setminus L$ to itself, so any cusp expansion of $\SS^3 \setminus L$ induces an embedded cusp expansion of $\orbQ_L$. For the remainder of the proof we will assume that lengths of parabolic elements are measured based on the preferred horoball packing of $\HH^3$, which by the preceding observation descends to embedded cusp expansions of both $\SS^3 \setminus L$ and $\orbQ_L$.

Let $\alpha$ be the multi-slope that we have Dehn filled along to obtain $\SS^3 \setminus K$ from $\SS^3 \setminus L$, let $p$ be the covering map from $\SS^3 \setminus K$ to $\orbQ$, and let $p_L$ be the covering map from $\SS^3 \setminus L$ to $\orbQ_L$. The slope $\alpha$ projects via $p_L$ to a slope $\alpha_p$, along which $\orbQ_L$ is filled to obtain $\orbQ$. Let $s$ be a Dehn surgery slope for $\orbQ$ that is not the quotient of the meridian corresponding to the knot complement. Then filling $\orbQ$ along $s$ is equivalent to filling $\orbQ_L$ along $(s,\alpha_p)$: 

\[
\xymatrix{
\SS^3 \setminus L \ar[d]^{p_L} \ar[r]^{\alpha} & \SS^3 \setminus K \ar[d]_{p} &\\
\orbQ_L \ar[r]^{\alpha_p} \ar@/_2pc/[rr]_{(s,\alpha_p)} & \orbQ \ar[r]^-{s} & \orbQ(s)=\orbQ_L(s,\alpha_p) 
}
\]

It therefore is enough to show that fillings of the form $(s, \alpha_p)$ of $\orbQ_L$ have no non-trivial exceptional surgeries that are good orbifolds. We claim that for every such multi-slope $(s, \alpha_p)$, each component has length greater than $6$. 

The length of each component of $\alpha_p$ is equal to the length of the component of $\alpha$ that it lifts to. By \cite[Theorem 3.10]{FP07} the length of the each component of $\alpha$ is at least $ \sqrt{6^2+1}>6$ since each twist region has at least $6$ crossings. Thus it remains to analyze the length of $s$. Applying \Cref{prop:quotientbases} with $G_1$ corresponding to the peripheral subgroup of $\orbQ_L$ and $G_2$ corresponding to the peripheral subgroup of the cover $\SS^3 \setminus L$ allows us to analyze the lengths of surgery slopes of the cusps of $\orbQ_L$ in terms of the slope lengths on the cusps of $\SS^3 \setminus L$. In particular, our hypotheses imply that the longitude of the planar cusp has length at least $18>16$. By \Cref{lem:bounds}, this results in $s$ having length greater than $6$. Combining both length bounds with \Cref{thm:orb_6_theorem}, this filling of $\orbQ_L$ is either hyperbolic or a bad orbifold. Hence, the corresponding filling of $\orbQ$ is either hyperbolic or a bad orbifold.
\end{proof}
We conclude with the following corollary, which is stated previously in the introduction. 

\begin{corn}[\ref{thm:only_knot_comm_class}]
\uniqueKnotComp
\end{corn}

\begin{proof}
By \Cref{thm:ht_no_rigid_cusps} $\SS^3 \setminus K$ does not cover a rigid cusped orbifold, and thus $\SS^3 \setminus K$ does not admit hidden symmetries (see \cite[Proposition 9.1]{NR92}). Denote by $\orbQ$ the (possibly trivial) quotient of $\SS^3 \setminus K$ by the group of symmetries that act freely on the cusp. Then by the main results of \cite{BBCW12} (especially \cite[Proposition 4.13]{BBCW12}),  each knot complement commensurable with $\SS^3 \setminus K$ (including $\SS^3 \setminus K$ itself) corresponds to an orbi-lens space filling of $\orbQ$. We have shown above in \Cref{lem:second_slope_long} that all non-trivial fillings of $\orbQ$ are hyperbolic or bad orbifolds. In particular, only the trivial filling of $\orbQ$, i.e., along the image of the meridian of $\SS^3\setminus K$, can result in an orbi-lens space filling.

Therefore, $\SS^3 \setminus K$  is the unique knot complement in its commensurability class.
\end{proof}

If we add the further assumption that $\SS^3 \setminus K$  does not admit any symmetries that act freely on the cusp, then $\SS^3 \setminus K$ would have to admit a lens space filling to be commensurable with a knot complement by \cite[Proposition 4.13]{BBCW12}. In this case, we can apply the bounds directly from Futer and Purcell \cite[Corollary 1.8]{FP07} with the Perelman's affirmative solution to the Geometrization Conjecture and obtain a similar result (see \cite{MorganTian2014} for example).
\begin{cor}\label{cor:unique_knot_comp_no_free_symms}
Let $\SS^3 \setminus K$ be an $(\epsilon,d_L)$-twisted and generic knot complement with at least $4$ twist regions  such that each twist region has at least $6$ crossings. If $\SS^3 \setminus K$ does not admit any symmetries that act freely on the cusp, then $\SS^3 \setminus K$ is the only knot complement in its commensurability class.
\end{cor}

\section{Quantifying Results}
\label{sec:QR}

Our main goal of this section is to address the following question:

\begin{question} Can we find quantifiable bounds on $\epsilon$ and the multi-slope $\alpha$ so that  $M=N(\alpha)$ is an $(\epsilon, d_{N})$-twisted  manifold (as defined in \Cref{def:edHTN})?
\end{question}

$(\epsilon, d_{N})$-twisted  manifolds are a generalization of $(\epsilon, d_{L})$-twisted knot complements.
Working in this broader context allows for a somewhat more streamlined discussion as we adapt various effectivization results of \cite{FPS19} to our setting. After adapting the necessary bounds, this section culminates in the construction of an infinite family of examples of  $\edHT$  and generic fillings in \Cref{subsec:examples-twisted-generic}.

Throughout this section, let $N$ be a hyperbolic manifold with cusps $\{C_0, C_1, \dots, C_n\}$, and let $M$ be a manifold obtained by Dehn filling all but one component of $N$. We may take the unfilled component to be $C_0$, so that our filling is along a multi-slope $\alpha = (-, s_1, \dots, s_n)$.  Here $s_i=\frac{p_i}{q_i}$ for integers $p_i,q_i \ne 0$ and $\gcd(p_i,q_i)=1$ and we use `$-$' to demarcate the unfilled component of the boundary of N. Given such an $N$, let $d_N=4\frac{vol(N)}{v_0}$, where just as before $v_0$ is the volume of the regular ideal tetrahedron.  

\begin{defin}\label{def:edHTN}
	We say that a (hyperbolic) cusped manifold $M$  is \define{$(\epsilon,d_N)$-twisted}, or alternatively that $M$ is an \define{$(\epsilon,d_N)$-twisted filling} of $N$, if for $0<\epsilon < 3.45$ the following three conditions hold:
\begin{enumerate}[label=(\roman*)]
\item $N_{\thick}$ is homeomorphic to $N \setminus n(\bigcup_i C_i)$, 
\item  $M_{\thick}$ is homeomorphic to $N_{\thick}$,  and
\item  $M_{\thickDN}$ is homeomorphic to $N_{\thick}$. 
\end{enumerate}
Furthermore, we say that $M$ is an \define{$(\epsilon,d_N)$-twisted and generic} filling of $N$ if it is $\edHTN$ and  no isometry of $M$ non-trivially permutes the set of core geodesics of the filling solid tori.

\end{defin}

\begin{prop}\label{prop:edNnonarithmetic}
If $M$ is an $(\epsilon,d_N)$ filling of $N$, then $M$ is non-arithmetic. Moreover, $d_N=4 \frac{vol(N)}{v_0}$ is an upper bound for the degree of any cover $p: M \rightarrow \orb$, where $\orb$ is an orientable orbifold.
\end{prop}

\begin{proof}
First we observe that $d_N=4\frac{vol(N)}{v_0} > 4 \frac{vol(M)}{v_0}$. By \cite{CM01} we have $vol(M)>2v_0$ so $d_N >8$. If $M$ is arithmetic, then by \cite[Theorem 4.6]{NR92} its shortest geodesic is at least 0.43137. Since $\frac{\epsilon}{d_N}<\frac{3.45}{8}=0.43125$, we have $M_{\thickDN}\cong M \not\cong N_{\thick}$, contradicting (iii). The final part of the claim follows directly from the minimum volume bounds established by \Cref{lem:min_vol_non_arithmetic}.
\end{proof}

Each of the three conditions of \Cref{def:edHTN} can be translated into restrictions on the geometry of $N$ and the multi-slope $\alpha$. Condition (i) demands that we choose $\epsilon$ smaller than  the systole length of $N$. Condition (ii) requires us to choose each filling slope large enough so that the core geodesics introduced under Dehn surgery are each shorter than $\epsilon$, while stabilizing the geodesic lengths coming from the unfilled manifold $N$. Condition (iii) further requires these core geodesics to be smaller than the  geodesics coming from  $N$ by a factor of  $d_N=4\frac{vol(N)}{v_0}$. Based on this description,  one can expect an explicit dependence on the systole length of $N$, the volume of $N$, and the multi-slope $\alpha$. Our goal is to use tools coming from the recent work of Futer--Purcell--Schleimer \cite{FPS19} (which builds off of the work of Hodgson and Kerckhoff \cite{HK05}) to make this relationship as explicit as possible. We now introduce some terminology and the results needed from Futer--Purcell--Schleimer.

Choose a cusp expansion for the cusps of $N$ corresponding to a maximal horoball packing in $\mathbb{H}^3$.  Under this cusp expansion, each cusp $C_{i}$ of $N$ has a torus boundary $\partial C_{i}$. For a slope $s_i$ on $\partial C_{i}$ define the \define{normalized length} of $s_i$ to be 
\begin{equation*}
\NL_{i} =  \frac{\ell_{\mathbb{E}}(s_i)}{\sqrt{area(\partial C_{i})}},
\end{equation*}
where $\ell_{\mathbb{E}}(s_i)$ is the length of the Euclidean geodesic isotopic to $s_i$. For a multi-slope ${\alpha = (-,s_1, \dots, s_n)}$, the \define{total normalized length $\NL$} is defined via 
\[ \frac{1}{\NL^2} = \sum_{i=1}^{n} \frac{1}{\NL_{i}^{2}}. \]

Let $\Sigma = \{ \gamma_{i} \}_{i=1}^{n}$ be the set of $n$ core geodesics introduced in $N(\alpha)$ via Dehn filling $N$. Let $\ell(\Sigma)$ be the total length of the $n$ geodesics in $\Sigma$, as measured in the complete metric on $M$. Our first step is showing that we can explicitly choose the multi-slope $\alpha$ as a function of $\epsilon$ and $vol(N)$ in such a way that $\Sigma \subset N(\alpha)_{<\epsilon/d_N}$. For this, we will need the following reformulation of \cite[Corollary 6.13]{FPS19}, specialized for our particular Dehn fillings. 

\begin{cor}[{\cite[Corollary 6.13]{FPS19}}] \label{cor:FPS1}
Let $M$ be the manifold obtained by filling $N$ along the multi-slope $\alpha = (-, s_1, \ldots, s_n)$. If  $\NL^{2} \geq 61.2$, then
\[ \frac{2\pi}{\NL^{2}+16.17} < \ell(\Sigma) < \frac{2\pi}{\NL^2 - 28.78} \]
\end{cor}

Intuitively, this lemma says that the total length of the core geodesics introduced from Dehn filling our crossing circle cusps is approximately $\frac{2\pi}{\NL^{2}}$, assuming our Dehn filling slopes are sufficiently long (in terms of normalized length). 

Since we only need to make sure that each $\gamma_{i} \in \Sigma$ is sufficiently short, we will only need the upper bound on $\ell(\Sigma)$. This requires an upper bound on $\frac{1}{\NL^2}$, given by the below lemma:

\begin{lem} \label{lem:ExpThin}
Let $N(\alpha)$ be the manifold obtained by Dehn filling $N$ along the multi-slope $\alpha = (-,s_1, \dots, s_n)$. For any $\epsilon >0$, if 
$$
\frac{1}{\NL^2} \leq \min \bigg\{ \frac{1}{61.2}, \frac{\epsilon v_0}{8\pi vol(N)+ 28.78 \epsilon v_0 } \bigg\},
$$
then $\Sigma \subset N(\alpha)_{<\epsilon/d_N}$.
\end{lem}

\begin{proof}
Since $\frac{1}{\NL^2} < \frac{1}{61.2}$, \Cref{cor:FPS1} applies, so $\ell(\Sigma) <\frac{2\pi}{\NL^2-28.78}$. Thus we just need to ensure that $\frac{2\pi}{\NL^2-28.78} \le \frac{\epsilon}{d_N} = \frac{\epsilon v_0}{4 vol(N)}$. Solving this equation for $\frac{1}{\NL^2}$ gives the second upper bound in the statement of the lemma.
\end{proof}

The above lemma explicitly shows that, for any choice of $\epsilon$, we can choose $\alpha$ so that our set of core geodesics $\Sigma$ is in  $N(\alpha)_{<\epsilon/d_N}$. We also want to make sure that these are the only geodesics in $N(\alpha)_{\thin}$ and quantify this relationship in terms of $\alpha$ and $\epsilon$. Another result of Futer--Purcell--Schleimer, restated for our purposes, is useful here.

\begin{thm}[{\cite[Theorem 1.2]{FPS19}}] \label{thm:FPS1}
Let $N(\alpha)$ be the manifold obtained by filling $N$ along the multi-slope $\alpha = (-, s_1, \dots, s_n)$. Fix $0 < \epsilon \leq \log(3)$ and $J >1$. If \[ \ell(\Sigma) \leq \min \bigg\{ \frac{\epsilon^{5}}{6771\cosh^{5}(0.6\epsilon + 0.1475)}, \frac{\epsilon^{5/2}\log(J)}{11.35} \bigg\}, \] then there exists a J-bilipschitz inclusion $\psi: N_{\thick} \hookrightarrow N(\alpha)_{\thick /1.2}$. 
\end{thm}

In our applications, we will always have $J \ge 1.001 $. In this case we let 
$$
C(\epsilon) = \frac{\epsilon^{5}}{6771\cosh^{5}(0.6\epsilon + 0.1475)},
$$
and observe that $C(\epsilon) \le  \frac{\epsilon^{5/2}\log(J)}{11.35}$ for any $\epsilon > 0$, $J\ge 1.001$, so we can disregard the latter bound. In what follows, let $\ell_{N}$ denote length as measured in the complete hyperbolic metric for $N$. Let $\ell_{N}(\gamma_s)$ denote the systole length of $N$, that is, the length of a shortest closed geodesic in $N$.

\begin{thm} \label{thm:quantify}
Let $N(\alpha)$ be the manifold obtained by filling $N$ along the multi-slope $\alpha = (-, s_1,\dots, s_n)$. If $\epsilon \leq \min \{ \frac{\ell_{N}(\gamma_s)}{1.001}, \log(3)\}$, and 
$$
\frac{1}{\NL^2} \leq \min \bigg\{  \frac{\epsilon v_0}{8\pi vol(N)+ 28.78 \epsilon v_0 }, \frac{1}{\frac{2\pi}{C(\epsilon)}+28.78} \bigg \},
$$
then  $M = N(\alpha)$ is $(\epsilon, d_{N})$-twisted. \end{thm}

\begin{proof}
Since $\epsilon < \ell_{N}(\gamma_s)$, item (1) from the definition $(\epsilon, d_N)$-twisted is satisfied. For items (2) and (3) to hold, we need to show that $\Sigma \subset N(\alpha)_{<\epsilon/d_N}$ and that these are the only closed geodesics in $N(\alpha)_{\thin}$. First, note that since $\epsilon \le \log(3) =  1.099..$, we have that  
$$
\frac{1}{\NL^2} \leq \frac{1}{\frac{2\pi}{C(1.099)}+28.78}  \leq 0.0000086 < \frac{1}{61.2}. 
$$
As a result, our conditions on $\alpha$ guarantee that the hypotheses of  \Cref{lem:ExpThin} hold, and so, $\Sigma \subset N(\alpha)_{<\epsilon/d_N}$. 

Let $\gamma \subset N$ be a closed geodesic and let $\gamma'$ be the corresponding geodesic in $N(\alpha)$. Note that since $\epsilon \le \frac{\ell_{N}(\gamma_s)}{1.001} < \ell_{N}(\gamma_s)$, $\gamma$ is in $N_{\thick}$. We need to show that $\gamma' \subset N(\alpha)_{\thick}$. To do this, we will apply  \Cref{thm:FPS1} with $J = 1.001$ and $\epsilon$ as given in the statement of the theorem. To satisfy the hypothesis of  \Cref{thm:FPS1} we need $\ell(\Sigma) \leq C(\epsilon).$ Similar to the proof of  \Cref{lem:ExpThin}, we use inequality (1), namely

$$
\ell(\Sigma) < \frac{2\pi}{\NL^2 - 28.78},  
$$
 to restate this condition in terms of our Dehn filling parameters. If we choose our multi-slope $\alpha = (-, s_1,\dots, s_n)$ so that $\frac{ 2\pi}{\NL^2-28.78} \leq C(\epsilon)$, then we have a $1.001$-bilipschitz map $\psi: N_{\thick} \hookrightarrow N(\alpha)_{\thick /1.2}$. It follows that we need 
$$
\frac{1}{\NL^2} \leq \frac{2}{\frac{2\pi}{C(\epsilon)}+28.78}.
$$
Under these conditions, our bilipschitz map guarantees that 
$$
\epsilon \le  \frac{\ell_{N}(\gamma_s)}{1.001}  \leq \frac{\ell_{N}(\gamma)}{1.001} \leq \ell_{N(\alpha)}(\gamma').
$$
\end{proof}

If we want to apply \Cref{thm:quantify} to particular examples of manifold $N$, it is useful to have a way to estimate $\NL_i^2$. A bit of plane geometry gives the below lemma:

\begin{lem}\label{lem:NL_bound}
Let $C_i$ be a cusp of $N$ with longitude $\lambda\in \RR_{\ge 0}$ and meridian $\mu=re^{i\theta}$, and let $s=p\mu+q\lambda$ be a slope on $\del C_i$. Then
$$
\NL_i^2 \ge \frac{p^2r^2+q^2\lambda^2}{r\lambda \sin \theta}-|pq \cot(\theta)|
$$	
\end{lem}

In this paper our main interest is in the case where $N$ is an FAL complement with a planar component $K_0$ and crossing circles $C_1,\dots ,C_n$, and the filling multi-slope has the form $\alpha=(-, \frac{1}{q_1},\dots, \frac{1}{q_n})$. For a crossing circle the longitude always has length 2, and the value of $\theta$ in the above lemma is bounded away from $0$ and $\pi$. This allows us to greatly simplify the bound in \Cref{lem:NL_bound}, as is shown by the following result of Purcell.

\begin{lem}[{\cite[Proposition 6.5]{Pur08b}}] \label{lem:NL}
Let $N(\alpha)$ be the knot compliment obtained by filling an FAL complement $N=\SS^3\setminus L$ along the multi-slope $\alpha = (-, \frac{1}{q_{1}}, \ldots, \frac{1}{q_{n}})$. Then the normalized length $\NL_{i}$ of the slope $\frac{1}{q_i}$ satisfies $\NL_{i} \geq \sqrt{2q_{i}}$.
\end{lem}

By combining \Cref{lem:NL} with the bound $vol(N)\le 10v_0(n-1)$ given in the appendix of \cite{Lac04}, we get the following corollary of \Cref{thm:quantify}:

\begin{cor} \label{thm:FALquantify}
Let $N(\alpha)$ be the knot complement obtained by filling an FAL complement $N=\SS^3\setminus L$ along the multi-slope $\alpha = (-, \frac{1}{q_1},\dots, \frac{1}{q_n})$. If $\epsilon \leq \min \{ \frac{\ell_{N}(\gamma_s)}{1.001}, \log(3)\}$, and 
$$
\sum_{i=1}^n\frac{1}{q_i^2} \leq \min \bigg\{  \frac{2\epsilon }{80\pi(n-1)+ 28.78 \epsilon}, \frac{2}{\frac{2\pi}{C(\epsilon)}+28.78} \bigg \},
$$
then  $M = N(\alpha)$ is $(\epsilon, d_{L})$-twisted. \end{cor}

While  \Cref{thm:FALquantify} does provide a quantified method to build $(\epsilon, d_{L})$-twisted  knot complements, there are many explicit dependencies required for this method: $\epsilon$ depends on the systole length of $N$ and $\alpha$ depends on both the systole length of $N$ and the volume of $N$ (which can be restated in terms of a dependence on the $n$ crossing circles). It is natural to ask if all of these dependencies are necessary. In fact, work of Meyer--Millichap--Trapp \cite{MeyerMillichapTrapp2018} shows that they are.

First, we note the dependence on volume  is necessary for \Cref{thm:FALquantify}. Recall from the proof of \Cref{prop:HT_non-empty} that $d_L=4\frac{vol(\SS^3\setminus L)}{v_0}$ was introduced to provide a bound on the degree of any cover from $N$ to an orbifold. Since there exist FAL complements that have both volume and orbifold covering degrees growing linearly with the number of crossing circles, we cannot avoid this dependence. The pretzel FAL complements discussed in \cite{MeyerMillichapTrapp2018}, exhibit this feature; see the proofs of Theorem 7.7 and Corollary 7.8 from their paper for a description of particular orbifold covers and volumes of these manifolds. 

Now, consider a pretzel FAL $L_n$ with $n$ crossing circles, no half-twist going through the first crossing circle, and a single half-twist going through each of the $n-1$ remaining crossing circles. Such an FAL always has a single planar component, and for $n \geq$ 3, $N_n = \SS^3 \setminus L_n$ is hyperbolic. The proof of \cite[Proposition 5.2]{MeyerMillichapTrapp2018} shows that $N_n$ has a closed geodesic $\gamma_n$ of length  $\ell(\gamma_n) = 2 \ln(\frac{\csc(\pi/n)+1}{\csc(\pi/n)-1})$. As $n \rightarrow \infty$, $\ell(\gamma_n) \rightarrow 0$. Thus, there exist FAL complements with arbitrarily short systole length. This implies that we cannot avoid the dependence of $\epsilon$ on the systole length of $N$ in the construction of an $(\epsilon, d_L)$-twisted knot complement.

Although the dependence on systole length is necessary when considering the full class of FAL complements, we can remove this dependence if we restrict to arithmetic FAL complements. This class includes all octahedral FALs (see \cite[Proposition 3.8]{Pur11} for an explicit description of this infinite class). By work of Neumann and Reid \cite[Corollary 4.5, Theorem 4.6]{NR92} any arithmetic link complement, in particular any arithmetic FAL complement has $\ell_{N}(\gamma_s) > 0.862554$. Using this fact, we get the following immediate corollary of \Cref{thm:FALquantify}, which gives a stronger version of that theorem, at the cost of restricting to arithmetic FALs.

\begin{cor} \label{thm:ArithQuant}
Let $N(\alpha)$ be the knot complement obtained by filling an arithmetic FAL complement $N=\SS^3\setminus L$ along the multi-slope $\alpha = (\frac{1}{q_{1}}, \ldots, \frac{1}{q_{n}})$. Let $\epsilon \le  \frac{0.86255}{1.001}<0.861688$ and 
$$
\sum_{i=1}^{n} \frac{1}{q_{i}} \leq \min \bigg\{  \frac{2\epsilon}{80\pi (n-1)+ 28.78 \epsilon }, \frac{2}{\frac{2\pi}{C(\epsilon)}+28.78} \bigg \}.
$$ 
Then  $M = N(\alpha)$ is $(\epsilon, d_{L})$-twisted. In particular, if we choose $\epsilon = 0.86168$ and we require  
$$
\sum_{i=1}^{n} \frac{1}{q_{i}} \leq \min  \bigg \{ \frac{1.72336}{80\pi(n-1) + 24.8}, 7.963 \times 10^{-6}  \bigg  \},
$$
then  $M = N(\alpha)$ is $(\epsilon, d_{L})$-twisted.
\end{cor}

\subsection{Examples of $\edHT$ and generic knot complements.}\label{subsec:examples-twisted-generic}

We conclude this section with a concrete family of $\edHT$ and generic knot complements. We will first discuss a rigorous construction and then contrast our numbers against experimental evidence which, although it may be susceptible to numerical errors, seems to suggest that $\edHT$ and generic knot complements are obtained via relatively low filling parameters.

\begin{figure}
\includegraphics[width=.35\textwidth]{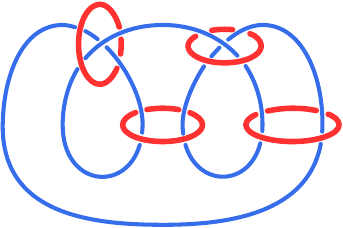}
\caption{\label{fig:OctFALexample} The octahedral FAL $L_4$, from which we exhibit $\edHT$ and generic knot complements that are the only knot complements in their respective commensurability classes.}
\end{figure}

We denote by $L_4$ the FAL shown in \Cref{fig:OctFALexample}. Using SnapPy's Sage interface \cite{CDGW09}, we can rigorously compute the orientation preserving symmetry group of $\SS^3 \setminus L_4$ and its action on the cusps. In particular, this group is isomorphic to $\ZZ_2\times \ZZ_2$, and it fixes each cusp of $\SS^3\setminus L_4$. From this it follows that any $\edHT$ filling of  $\SS^3 \setminus L_4$ is automatically an $\edHT$ and generic filling.

Given that the link complement  $\SS^3 \setminus L_4$ is in SnapPy's octahedral census as $\mathtt{ooct06\_06059}$, we see that $\SS^3 \setminus L_4$ is an octahedral FAL complement, i.e., it is arithmetic and has invariant trace field $\QQ(\I)$. In this case \cite[Corollary 4.5, Theorem 4.6]{NR92} gives a lower bound on the systole of $0.962424$, which is an improvement on the bound used for \Cref{thm:ArithQuant}. While we could plug this value directly into \Cref{thm:FALquantify}, we can get a smaller lower bound on the $q_i$ by using \Cref{thm:quantify}. In particular, since SnapPy can rigorously compute the cusp translations for $\SS^3\setminus L_4$, we can use \Cref{lem:NL_bound} to get a better bound on $\NL_i^2$ then is provided by \Cref{lem:NL}. As the link complement is octahedral, we can apply exact values from $\ZZ[i]$ with  SnapPy's combinatorial calculations (eg the computation of peripheral tori) to rigorously obtain the following inequalities: $0\le \cot\theta_i\le 1$, $r_i\lambda_i\sin\theta \le 6$, and $r_i\ge 1$. Given that $p_i=1$ for all $i$ and crossing circles always have longitudes of length $\lambda_i=2$, \Cref{lem:NL_bound} yields
$$
\NL_i^2 \ge \frac{1+4q_i^2}{6}-q_i,
$$
which implies that $\frac{1}{\NL_i^2} \le \frac{3}{2(q_i-1)^2}$ since we may assume that $q_i\ge 3$.

To apply \Cref{thm:quantify}, we find that we need $\frac{1}{\NL^2} \le \sum_{i=1}^n \frac{3}{2(q_i-1)^2} \le .0000057524$ to ensure that the resulting filling of $\SS^3\setminus L_4$ is $\edHT$. It follows that filling along slopes $\frac{1}{q_i}$ with $q_i\ge 1023$ for all $i$ will yield an $\edHT$ knot complement, which will moreover be $\edHT$ and generic given the symmetry group calculation for $\SS^3\setminus L_4$.

All calculations above are rigorous, and thus by \Cref{cor:unique_knot_comp_no_free_symms} the knots constructed above are an infinite family of knots that are the only knots in their commensurability classes. Furthermore, these are a new family that has not previously appeared in the literature, as we now show. By \cite{FKP08}, the constructed knots all have volumes within $1.18\times 10^{-11}$ of $vol(\SS^3\setminus L_4)=21.9831742603..$. There are three families of knots in the literature (that we are aware of) for which each is the only knot in its commensurability class. In \cite{MaMa08}, the authors show that for $n\ne 7$, a $(-2,3,n)$ pretzel knot complement is the unique knot complement in its commensurability class. Such knot complements are all obtained by surgery on the arithmetic two-component census link $\mathtt{L9n9}$. Thus their volumes are bounded by $vol(\mathtt{L9n9})=5.333..$, so our knots are distinct from these. In \cite{Mil17} it is shown that certain pretzel knots obtained by long Dehn fillings of a subfamily of FALs have no other knots in their respective commensurability classes. Again using \cite{FKP08}, these knots must have volumes close to their FAL ancestors, whose volumes are all at least $29.31$, so our knots are distinct from these as well. Finally, in \cite{RW08} the authors show that 2-bridge knots have no other knots in their commensurability classes, and that such knots always admit non-trivial symmetries acting freely on the cusp. Since no symmetry of $\SS^3\setminus L_4$ acts freely on the planar cusp, it follows that no symmetry of an $\edHT$ filling will act freely on the cusp, so our knots are not 2-bridge knots. Thus the family of knots we construct is distinct from the three known families above, which gives the following theorem:

\begin{thm}\label{thm:concrete_fam}
Let $L_4$  be the link in \Cref{fig:OctFALexample}. Let $\SS^3 \setminus K$ be a knot complement obtained by preforming $1/q_i$ Dehn filling on each crossing circle cusp of $\SS^3 \setminus L_4$. If each $q_i\geq 1023$, then $\SS \setminus K$ is an $\edHT$ and generic filling and therefore it admits no hidden symmetries and is the unique knot complement in its commensurability class.
\end{thm}

 While the bound of $q_i\ge 1023$ is as far as we are able to push things with rigorous computations, the following non-rigorous computations suggest that a far more modest bound will suffice. Using SnapPy, we find that filling the crossing circle cusps of $L_4$ along slopes $\frac{1}{22}$ yields a knot complement $M$ whose 4 shortest geodesics all have length less than $.02$. All other geodesics of $M$ have length at least $1.76$. Since the volume of $\SS^3\setminus L_4$ is $\sim 21.9831$, and the systole is $\sim 1.76$, if we choose $\epsilon = 1.75$ then we get $\frac{\epsilon}{d_L} > .02$. Thus the four short geodesics of $M$ are in $M_{<\epsilon/d_L}$, and these are the only geodesics in $M_{\thin}$. Since drilling out these geodesics yields a manifold homeomorphic to $\SS^3\setminus L_4$, it follows that $M$ is $\edHT$ and generic. Although systole computations and drilling geodesics are non-rigorous in SnapPy, these calculations suggest that for $L_4$, $q_i \ge 22$ is likely large enough.

\bibliographystyle{alpha2}
\bibliography{biblio}

\end{document}